\documentclass[a4paper,10pt]{amsart}
\usepackage{amsmath,amssymb, amsbsy}
\usepackage{subfigure}
\usepackage{color,psfrag}
\usepackage{enumerate}
\usepackage[dvips]{graphicx}

\usepackage{latexsym}
\usepackage{a4wide}
\usepackage{amscd}
\usepackage{mathrsfs}

\newcommand{\R}{{\mathbb R}}
\newcommand{\N}{{\mathbb N}}

\newcommand{\SN}{{\mathbb S}^{N-1}}

\newcommand{\e }{\varepsilon}
\newcommand{\Di}{{\mathcal D}^{1,2}(\R^N)}

\renewcommand{\geq }{\geqslant}

\renewcommand{\leq }{\leqslant}
\newcommand{\abs}[1]{\left\vert #1 \right\vert}

\newenvironment{pf}{\noindent{\sc Proof}.\enspace}{\hfill\qed\medskip}
\newenvironment{pfn}[1]{\noindent{\bf Proof of
    {#1}.\enspace}}{\hfill\qed\medskip}
\newtheorem{Theorem}{Theorem}[section]
\newtheorem{Corollary}[Theorem]{Corollary}
\newtheorem{Lemma}[Theorem]{Lemma}
\newtheorem{Proposition}[Theorem]{Proposition}
\theoremstyle{definition}

\newtheorem{remark}[Theorem]{Remark}

\begin{document}

\title[Singularity of eigenfunctions, Part II]{Singularity of eigenfunctions at the junction of shrinking tubes, Part II}
\author{Laura Abatangelo, Veronica Felli, Susanna Terracini}

\address{
\hbox{\parbox{5.7in}{\medskip\noindent
  L. Abatangelo, V. Felli\\
Dipartimento di Matematica e Applicazioni,\\
 Universit\`a di Milano Bicocca, \\
Piazza Ateneo Nuovo, 1, 20126 Milano (Italy)         . \\[2pt]
         {\em{E-mail addresses: }}{\tt laura.abatangelo@unimib.it, veronica.felli@unimib.it.}\\[5pt]
 S. Terracini\\
 Dipartimento di Matematica ``Giuseppe Peano'',\\
Universit\`a di Torino, \\
Via Carlo Alberto, 10,
10123 Torino (Italy). \\[2pt]
                                     \em{E-mail address: }{\tt susanna.terracini@unito.it.}}}
}

\date{February 12, 2013}

\thanks{
\indent Partially supported by the PRIN2009 grant ``Critical Point Theory and 
Perturbative Methods for Nonlinear \\
\indent Differential Equations''.}

\begin{abstract}
In continuation with \cite{FT12}, we investigate the asymptotic behavior of weighted eigenfunctions
in two half-spaces connected by a thin tube. We provide several improvements about some convergences
stated in \cite{FT12}; most of all, we provide the exact asymptotic behavior of the implicit normalization for
solutions given in \cite{FT12} and thus describe  the
$(N-1)$-order singularity developed at a junction of the tube
(where $N$ is the space dimension).
\end{abstract}

\maketitle

\section{Introduction and statement of the main result}

The interest in the spectral analysis of thin branching domains
arising in the theory of quantum graphs modeling waves in thin
graph-like structures (narrow waveguides, quantum wires, photonic
crystals, blood vessels, lungs), see e.g. \cite{CF, kuchment},
motivates a large literature dealing with elliptic eigenvalue problems
in varying domains; we mention among others
\cite{anne,AD,arrieta,AK,BV,BHM,BZ,CH,DD,Dancer1,Dancer2,daners,grigorieff,RT,stummel}.
 
In \cite{FT12}, the asymptotic behavior of eigenfunctions at the
junction of shrinking tubes has been investigated.  In a
 dumbbell domain which is
going to disconnect, it can be shown that, generically, the mass
of a given eigenfunction of the Dirichlet Laplacian concentrates in
only one component of the limiting domain, while the restriction to the other domain, when
suitably normalized, develops a singularity at the junction of the
tube, as the channel section tends to zero.  The main result of
\cite{FT12} states that, under a proper nondegeneracy condition, the
normalized limiting profile has a singularity of order $N-1$, where
$N$ is the space dimension. The strategy developed in \cite{FT12} to evaluate the rate to the
singularity at the junction is based upon a sharp control of the
transversal frequencies along the connecting tube, inspired by the monotonicity method introduced by Almgren
\cite{almgren} and then extended by Garofalo and Lin \cite{GL}
to elliptic operators with variable coefficients in order to prove
unique continuation properties.

In continuation with \cite{FT12}, we investigate the asymptotic
behavior of solutions to weighted eigenvalue problems in a dumbbell
domain $\Omega^\e\subset\R^N$, $N\geq 3$, formed by two
half-spaces connected by a tube with length $1$ and cross-section of
radius $\e$:
$$
\Omega^\e=D^-\cup \mathcal C_\e\cup D^+,
$$
where    $\e\in (0,1)$ and
\begin{align*}
D^-&=\{(x_1,x')\in \R\times \R^{N-1}:x_1<0\},\\
\mathcal C_\e&=\{(x_1,x')\in \R\times \R^{N-1}:0\leq x_1\leq1,\ |x'|<\e\},\\
D^+&=\{(x_1,x')\in \R\times \R^{N-1}:x_1>1\}.
\end{align*}
\begin{figure}[h]
 \centering
   \begin{psfrags}
     \psfrag{D-}{$D^-$}
     \psfrag{D+}{$D^+$}
\psfrag{e}{${\scriptsize{\e}}$}
\psfrag{1}{${\scriptsize{1}}$}
\psfrag{C}{$\mathcal C_\e$}
     \includegraphics[width=7cm]{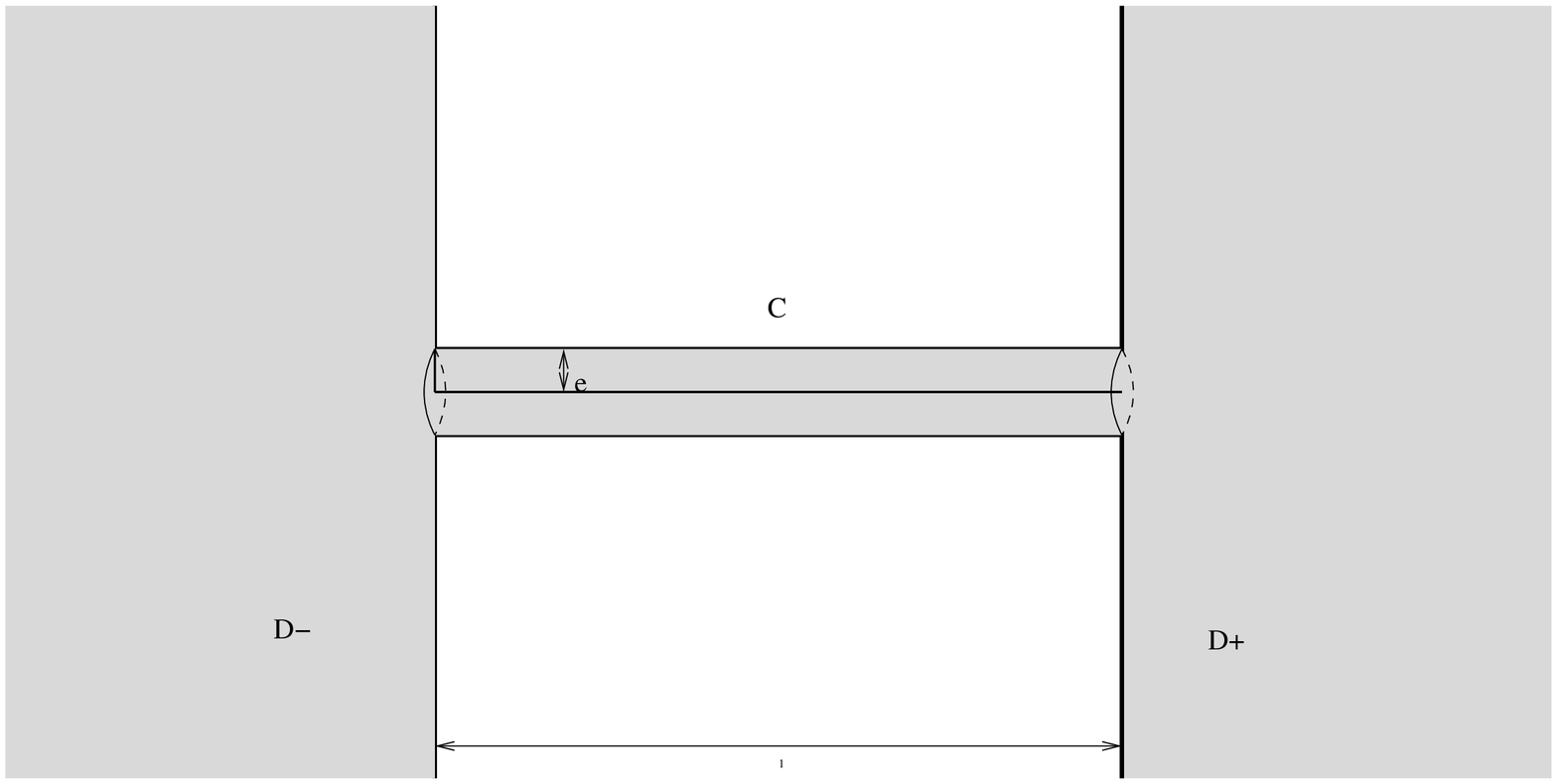}
   \end{psfrags}
 \caption{The domain $\Omega^\e$.}\label{fig:dd}
\end{figure}

\noindent We also denote, for all $t>0$,
$$
B^+_t:=D^+\cap B({\mathbf e}_1,t),\quad
B^-_t:=D^-\cap B({\mathbf 0},t),
$$
where ${\mathbf e}_1
=(1,0,\dots,0)\in \R^N$, ${\mathbf 0}=(0,0,\dots,0)$, and $B(P,t):=\{x\in\R^N:|x-P|<t\}$ denotes the
ball of radius $t$ centered at $P$.
Let $p\in C^1(\R^N,\R)\cap L^{\infty}(\R^N)$ be a weight  satisfying
\begin{align}
\label{eq:p} & p\geq 0\text{ a.e. in }\R^N,\
p\in L^{N/2}(\R^N),\ \nabla p(x)\cdot x\in L^{N/2}(\R^N),
\ \frac{\partial p}{\partial x_1}\in L^{N/2}(\R^N),\\
\label{eq:p2} & p\not\equiv 0\text{ in }D^-,\quad 
 p\not\equiv 0\text{ in }D^+,\quad 
p(x)=0\text { for all } x\in B^-_3\cup \mathcal C_1
\cup B^+_3.
\end{align}
Assumption \eqref{eq:p2} is stronger than in \cite{FT12}. We are
confident that the present arguments apply even under the weaker
assumption of \cite{FT12}, up to
several modifications mainly concerning calculus. For reader's
convenience we consider worthwhile presenting the argument in this
simpler case.

 For every open set $\Omega\subset\R^N$,  we denote as $\sigma_p(\Omega)$ the
 set of the diverging eigenvalues $\lambda_1(\Omega)\leq
\lambda_2(\Omega)\leq\cdots\leq\lambda_k(\Omega)\leq \cdots$
 (where each $\lambda_k(\Omega)$ is repeated as many times as its multiplicity) of 
 the weighted eigenvalue problem 
$$
\begin{cases}
-\Delta \varphi=\lambda p \varphi,&\text{in }\Omega,\\
\varphi=0,&\text{on }\partial \Omega.
\end{cases}
$$
It is easy to verify that $\sigma_p(D^-\cup D^+)=\sigma_p(D^-)\cup \sigma_p(D^+)$. 

Let us assume that there exists $k_0\geq 1$ such that 
\begin{align}
  \label{eq:53} \lambda_{k_0}(D^+)&\text{ is simple and the
    corresponding eigenfunctions}\\
\notag&\text{ have in ${\mathbf e}_1
    $ a zero of order $1$},\\
  \label{eq:54} \lambda_{k_0}(D^+)&\not\in \sigma_p(D^-).
  \end{align}
We can then fix an eigenfunction $u_0\in{\mathcal
    D}^{1,2}(D^+)\setminus\{0\}$ associated to $\lambda_{k_0}(D^+)$,
  i.e. solving
\begin{equation}\label{eq:u0}
\begin{cases}
-\Delta u_0=\lambda_{k_0}(D^+) p u_0,&\text{in }D^+,\\
u_0=0,&\text{on }\partial D^+,
\end{cases}
\end{equation}
such that 
\begin{equation}\label{eq:13}
\frac{\partial u_0}{\partial x_1}({\mathbf e}_1
)>0.
\end{equation}
Here and in the sequel, for every open set $\Omega\subseteq\R^N$, ${\mathcal
    D}^{1,2}(\Omega)$ denotes the functional space obtained as completion of 
$C^\infty_{\rm c}(\Omega)$ 
with respect to the Dirichlet norm $\big(\int_{\Omega}|\nabla u|^2dx\big)^{1/2}$.

From \cite[Example 8.2, Corollary 4.7, Remark 4.3]{daners} (see also
\cite[Lemma 1.1]{FT12}), it follows that, letting
\begin{equation*}
\lambda_\e=\lambda_{\bar k}(\Omega^\e)  
\end{equation*}
where $\bar k=k_0+
\mathop{\rm card}\big\{
j\in\N\setminus\{0\}:\lambda_j(D^-)\leq \lambda_{k_0}(D^+)\}$,
so that $\lambda_{k_0}(D^+)=\lambda_{\bar k}(D^-\cup D^+)$, there
holds 
\begin{equation}\label{eq:52}
\lambda_\e\to \lambda_{k_0}(D^+)\quad\text{as }\e\to0^+.
\end{equation}
Furthermore, for every $\e$ sufficiently small, $\lambda_\e$
is simple and there exists an eigenfunction $u_\e$
associated to $\lambda_\e$, i.e. satisfying
\begin{equation}\label{problema}
\begin{cases}
-\Delta u_\e=\lambda_\e p u_\e,&\text{in }\Omega^\e,\\
u_\e=0,&\text{on }\partial \Omega^\e,
\end{cases}
\end{equation}
such that
\begin{equation}\label{convergenza_u_0}
u_\e\to u_0\quad\text{in }{\mathcal D}^{1,2}(\R^N)
\quad\text{as }\e\to0^+,
\end{equation}
where in the above formula we mean the functions $u_\e,u_0$ to be trivially extended to the whole $\R^N$.
We refer to \cite[\S 5.2]{bucur2006} for uniform convergence of eigenfunctions.

For all $t>0$, let us denote as $\mathcal H_t^-$ the completion of 
$C^\infty_{\rm c}(D^-\setminus B_t^-)$ with
respect to the norm $\big(\int_{D^-\setminus {B_t^-}}|\nabla
v|^2dx\big)^{1/2}$, i.e.  $\mathcal H_t^-$ is the space of functions
with finite energy in $D^-\setminus \overline{B_t^-}$ vanishing on
$\partial D^-$.
We recall that functions in $\mathcal H_t^-$ satisfy the following Sobolev
type inequality
\begin{equation}\label{eq:sobHt-}
C_S\bigg(\int_{D^-\setminus {B_t^-}}|v(x)|^{2^*}\!dx\bigg)^{\!\!2/2^*}\!\!\!\leq 
\int_{D^-\setminus {B_t^-}}\!|\nabla v(x)|^2dx,\text{ for all
}t>0\text{ and }v\in \mathcal H_t^-,
\end{equation}
for some $C_S=C_S(N)>0$ depending only on the dimension $N$ (and
independent on $t$), see \cite[Lemma 3.2]{FT12}.

We also define, for all $t>0$,
\begin{equation}\label{eq:defGamma_r-}
\Gamma_t^-=D^-\cap \partial  B^-_{t}.
\end{equation}
Let
\begin{align*}
  \Psi: {\mathbb S}^{N-1}\to\R, \quad \Psi (\theta_1,\theta_2,\dots,\theta_N)=\frac{\theta_1}{\Upsilon_N},
\end{align*}
being ${\mathbb S}^{N-1}=\{(\theta_1,\theta_2,\dots,\theta_N)\in
\R^N:\sum_{i=1}^N\theta_i^2=1\}$ the unit $(N-1)$-dimensional sphere
and
\begin{align}\label{eq:upsilonN}
\Upsilon_N=\sqrt{\tfrac12{\textstyle{\int}}_{{\mathbb
        S}^{N-1}}\theta_1^2d\sigma(\theta)}.
\end{align}
Here
and in the sequel, the notation $d\sigma$ is used to denote the volume
element on $(N-1)$-dimensional surfaces.  
We
notice that, letting  
\begin{align*}
&{\mathbb
  S}^{N-1}_-:=\{\theta=(\theta_1,\theta_2,\dots,\theta_N)\in{\mathbb
  S}^{N-1}:\theta_1<0\},\\
&{\mathbb
  S}^{N-1}_+:=\{\theta=(\theta_1,\theta_2,\dots,\theta_N)\in{\mathbb
  S}^{N-1}:\theta_1>0\},
\end{align*}
$\Psi^-=-\frac{\theta_1}{\Upsilon_N}$ is the
first positive $L^2({\mathbb S}^{N-1}_-)$-normalized eigenfunction of
$-\Delta_{{\mathbb S}^{N-1}}$ on ${\mathbb S}^{N-1}_-$ under null
Dirichlet boundary conditions satisfying
\begin{equation}\label{eq:eigenpsi-}
-\Delta_{{\mathbb
    S}^{N-1}}\Psi^-=(N-1) \Psi^-\quad\text{on }{\mathbb S}^{N-1}_-,
\end{equation}
and 
$\Psi^+=\frac{\theta_1}{\Upsilon_N}$ is the
first positive $L^2({\mathbb S}^{N-1}_+)$-normalized eigenfunction of
$-\Delta_{{\mathbb S}^{N-1}}$ on ${\mathbb S}^{N-1}_+$ under null
Dirichlet boundary conditions satisfying
\begin{equation}\label{eq:eigenpsi+}
-\Delta_{{\mathbb
    S}^{N-1}}\Psi^+=(N-1) \Psi^+\quad\text{on }{\mathbb S}^{N-1}_+.
\end{equation}
The main results of \cite{FT12} are summarized in the following theorem.
\begin{Theorem}\label{t:main}{\rm (\cite{FT12})}
  Let us assume \eqref{eq:p}--\eqref{eq:13} hold and let $u_\e$ as in
  \eqref{problema}. Then there exists $\tilde k\in(0,1)$ such that,
  for every sequence $\e_{n}\to 0^+$, there exist a subsequence
  $\{\e_{n_j}\}_j$, $U\in C^2(D^-)\cup\big(\bigcup_{t>0}\mathcal
  H_t^-\big)$, $U\not\equiv 0$, and $\beta<0$ such that
\begin{align}
\tag{i}
\frac{u_{\e_{n_j}}}{\sqrt{\int_{\Gamma^-_{\tilde
        k}}u_{\e_{n_j}}^2d\sigma}}\to U  \ \text{ as }j\to+\infty
\quad &\text{strongly in }\mathcal H_t^- \text{ for every }t>0 \text{ and in }\\[-10pt]
& \notag 
C^2(\overline{B_{t_2}^-\setminus B_{t_1}^-}) \text{ for all }0<t_1<t_2;\\
\tag{ii} \lambda^{N-1} U(\lambda x)\to \beta\,\frac{x_1}{|x|^N}
\ \text{ as }\lambda \to 0^+
\quad&\text{strongly in }\mathcal H_t^- \text{ for every }t>0 \text{ and in }\\[-5pt]
& \notag
C^2(\overline{B_{t_2}^-\setminus B_{t_1}^-}) \text{ for all }0<t_1<t_2.
\end{align}
 \end{Theorem}
The aim of the present paper is twofold. On one hand, we will remove the
dependence on the subsequence in the previous statement. On the
other hand, the aforementioned theorem provides an implicit
normalization (i.e. $\int_{\Gamma^-_{\tilde
k}}u_{\e_{n_j}}^2d\sigma$) for the sequence of solutions to detect the limit
profile; we will determine the exact
behavior of this normalization,
thus providing an asymptotics of eigenfunctions, which will turn out
to be  independent of $\widetilde k\in (0,1)$. To this aim, we
proceed step by step, analyzing the asymptotics at succeeding
points, starting at the right junction where an initial
normalization is given by \eqref{convergenza_u_0} and \eqref{eq:13}.
In view of \cite[Section 4]{AT12}, the final behavior of
$\int_{\Gamma^-_{\tilde k}}u_{\e_{n_j}}^2d\sigma$ will depend on the
particular domain's shape, which will be recognizable by some
coefficients appearing in the leading term of the asymptotic
expansion. More precisely, information about the geometry may be
discerned in the dependence of the coefficients on the limit profiles
produced by a  blow-up at those points where a drastic
change of geometry occurs. 

We believe that from the  asymptotics of eigenfunctions proved in the
present paper an exact estimation of the rate of convergence of
eigenvalues on the perturbed domain to eigenvalues on the limit domain
could follow; this is the object of a current investigation.

Before stating our main result, let us introduce the functions
describing  the domain's geometry after blowing-up at each junction. 
Let us denote 
\begin{equation*}
 \widetilde D=D^+\cup T_1^-, \quad T_1^-=\{(x_1,x'):|x'|<1,\ x_1\leq1\}.
\end{equation*}
In \cite[Lemma 2.4]{FT12}, it is proved that there exists a unique
function $\Phi$ satisfying 
\begin{equation}\label{eq_Phi_1}
\begin{cases}
\int_{T_1^-\cup B^+_{R-1}}\Big(|\nabla \Phi(x)|^2
+|\Phi(x)|^{2^*}\Big)
\,dx<+\infty\text{ for all }R>2,\\[5pt]
-\Delta \Phi=0\text{ in a distributional sense in }\widetilde D,
\quad \Phi=0\text{ on }\partial \widetilde D,\\[5pt]
\int_{D^+}|\nabla (\Phi-(x_1-1))(x)|^2\,dx<+\infty.
\end{cases}
\end{equation}
Furthermore $\Phi>0$ in $\widetilde D$ and, by \cite[Lemma 2.9]{FT12}, there holds 
\begin{equation}\label{eq:Phiinfty}
  \Phi(x)=(x_1-1)^++O(|x-{\mathbf e}_1|^{1-N})\quad\text{in $D^+$ as
  }|x-{\mathbf e}_1|\to +\infty.
\end{equation}
Let us define
\begin{align}\label{eq:hatD}
 &\widehat D=D^-\cup T_1^+, \quad T_1^+=\{(x_1,x'):|x'|<1,\ x_1\geq0\},\\
\label{eq:T1}&\Sigma=\{x'\in\R^{N-1}:|x'|<1\},\quad 
T_1=\{(x_1,x'):x_1\in\R,\ |x'|<1\}.
\end{align}
We  denote as $\lambda_1(\Sigma)$  the first eigenvalue of the Laplace
operator on $\Sigma$ under null Dirichlet boundary conditions and
as $\psi_1^\Sigma(x')$  the corresponding positive
$L^2(\Sigma)$-normalized eigenfunction, so that 
\begin{equation}\label{eq:eigensigma}
\begin{cases}
-\Delta_{x'}\psi_1^\Sigma(x')=
\lambda_1(\Sigma)\psi_1^\Sigma(x'),&\text{in }\Sigma,\\
\psi_1^\Sigma=0,&\text{on }\partial\Sigma,
\end{cases}
\end{equation}
being $\Delta_{x'}=\sum_{j=2}^N\frac{\partial^2}{\partial x_j^2}$,
$x'=(x_2,\dots,x_N)$.  We define 
\begin{equation*}
h:T_1\to\R,
\quad h(x_1,x')=e^{\sqrt{\lambda_1(\Sigma)}x_1}\psi_1^\Sigma(x'),
\end{equation*}
and observe that $h\in C^2(T_1)\cap C^0(\overline{T_1})$ satisfies 
\begin{equation*}
\begin{cases}
-\Delta h=0,&\text{in }T_1,\\
h=0,&\text{on }\partial T_1.
\end{cases}
\end{equation*}
In \cite[Lemma 2.7]{FT12} it is proved that there exists a unique
function $\widehat \Phi:\widehat  D\to\R$ such that 
\begin{equation}\label{eq_Phi_hat}
\begin{cases}
\int_{D^-}\Big(|\nabla \widehat \Phi(x)|^2
+|\widehat \Phi(x)|^{2^*}\Big)
\,dx<+\infty,\\[5pt]
-\Delta \widehat \Phi=0\text{ in a distributional sense in }\widehat  D,
\quad \widehat \Phi=0\text{ on }\partial \widehat D,\\[5pt]
\int_{T_1}|\nabla (\widehat \Phi-h)(x)|^2\,dx<+\infty.
\end{cases}
\end{equation}
Furthermore 
\begin{equation}\label{eq:prop_hat_phi}
\widehat\Phi>0 \text{ in }\widehat D,
\quad \widehat\Phi\geq h \text{ in }T_1,
\quad \widehat\Phi-h\in
{\mathcal D}^{1,2}(\widehat D),
  \end{equation}
and, by \cite[Lemma 2.9]{FT12},
\begin{equation}\label{eq:stimahatPhiinfty}
  \widehat\Phi(x)=O(|x|^{1-N})\quad\text{as }|x|\to+\infty,\, x\in D^-.
\end{equation}
 A further limiting profile which plays a role in the asymptotic
behavior of eigenfunctions $u_\e$ at the singular junction is provided
in the following lemma.

\begin{Lemma}\label{l:barU}
If \eqref{eq:p}--\eqref{eq:13} hold, 
there exists a unique
function $\overline  U:D^-\to\R$  such that 
\begin{equation}\label{problema_U}
\begin{cases}
\overline  U\in \bigcup_{R>0}C^2(\overline{D^-\setminus
  B_R^-}),
&\overline U\in \bigcup_{R>0}\mathcal H_R^-,\\
-\Delta \overline U=\lambda_{k_0}(D^+) p \,\overline U, &\text{in
}D^-,\\
\overline U =0, &\text{on }\partial D^-\setminus\{{\mathbf 0}\},\\[5pt]
\lambda^{N-1}\overline U(\lambda
\theta)\mathop{\longrightarrow}\limits_{\lambda\to 0^+} \Psi^-(\theta),&\text{in }C^0({\mathbb S}^{N-1}_-).
\end{cases}
\end{equation}
\end{Lemma}
Our main result is the following theorem describing the behavior as $\e\to 0^+$ of $u_\e$  at the junction
${\mathbf 0}=(0,\dots,0)$.  

\begin{Theorem}\label{t:teorema_principale}
 Let us assume
  \eqref{eq:p}--\eqref{eq:13} hold
and let $u_\e$ as in \eqref{problema}.
Then 
\begin{equation}
  \frac{ e^{\frac{\sqrt{\lambda_1(\Sigma)}}{\e}}\, u_\e}{\e^{N}}
  \to\bigg(\int_{{\mathbb S}^{N-1}_-}
\widehat\Phi(\theta)\Psi^-(\theta)d\sigma\bigg) 
  \bigg(\int_{\Sigma}\Phi(1,x')\psi_1^\Sigma(x')dx'\bigg)
\left(\frac{\partial u_0}{\partial x}({\mathbf e}_1)\right)
 \overline U
\end{equation}
as $\e \to 0^+$ strongly in $\mathcal H_t^-$ for every $t>0$ and in
$C^2(\overline{B_{t_2}^-\setminus B_{t_1}^-})$ for all $0<t_1<t_2$,
where $\Phi$ and  $\widehat \Phi$ are defined in 
 \eqref{eq_Phi_1} and \eqref{eq_Phi_hat}
respectively, and $\overline U$ is as in Lemma \ref{l:barU}.
\end{Theorem}

The paper is organized as follows. In section \ref{sec:indep-subs}  we improve 
Theorem \ref{t:main} ruling out dependance on subsequences and prove
Lemma \ref{l:barU}
completely classifying the limit profile at the left
junction. In section \ref{sec:asympt-behav-norm} we describe the
asymptotic behavior of the normalization of Theorem \ref{t:main} and
prove Theorem \ref{t:teorema_principale}; to
this aim we first evaluate the asymptotic
behavior of the denominator of the Almgren quotient at a fixed point
in the corridor, then at $\e$-distance from the left junction in the
corridor, and finally at a fixed distance from the left junction in
$D^-$.

%%%%%%%%%%%%%%%%%%%%%%%%%%%%%%%%%%%%%%%%%%%%%%%%%%%%%%%%%%%%%%%%%%%%%%%%%%%%%%%%%%%%%%%%%%%%%%%%%%%%%%%

\section{Independence of the subsequence}\label{sec:indep-subs}

A deep insight into \cite{FT12} highlights how the dependence on the subsequences
in Theorem \ref{t:main} is \textit{a priori} given by two different
facts: on one hand, the convergence to the limit profile in the
blow-up analysis at the right junction  up to subsequences and, on the other hand, the possible multiplicity of the
limit profiles at the left junction (named $U$ throughout \cite{FT12}).
In this section we rule out both such occurrences.

\subsection{Independence in the blow-up limit on the right}

The first improvement concerns Lemma 4.1 in \cite{FT12}. We recall
some notation for the sake of clarity.

\noindent Let us define
\begin{align}\label{eq:84}
\widetilde u_\e:\widetilde \Omega^\e\to\R,\quad
\widetilde u_\e(x)=\frac1\e u_\e\big({\mathbf e}_1
+\e(x-{\mathbf e}_1
)\big),
\end{align}
where
\begin{equation}\label{eq:85}
\widetilde \Omega^\e:={\mathbf e}_1+\frac{\Omega^\e-{\mathbf
    e}_1}{\e} =\{x\in\R^N:{\mathbf e}_1 +\e(x-{\mathbf e}_1)\in\Omega^\e\}.
\end{equation}
For
all $R>1$, let  
$\mathcal H^+_R$  be the completion of 
$C^\infty_{\rm c}\big(\big((-\infty,1)\times\R^{N-1}\big)\cup \overline{B_R^+}\big)$
with respect to the norm $\big(\int_{((-\infty,1]\times\R^{N-1})\cup
  B_R^+}|\nabla v|^2dx\big)^{1/2}$, i.e.  $\mathcal H_R^+$ is the space of functions with
finite energy in $((-\infty,1]\times\R^{N-1})\cup \overline{B_R^+}$ vanishing on
$\{(1,x')\in \R\times\R^{N-1}:|x'|\geq R\}$.

\begin{Lemma}\label{l:tilde_u_eps_bound} {\rm (\cite[Lemma 4.1 and Corollary
   4.4]{FT12})}
 For every sequence $\e_{n}\to 0^+$ there exist a subsequence
 $\{\e_{n_k}\}_k$ and a constant $\widetilde C>0$ such that
 $\widetilde u_{\e_{n_k}}\to\widetilde C \Phi$ strongly in $\mathcal
 H_R^+$ for every $R>2$ and  in  $C^2(\overline{B_{r_2}^+\setminus
  B_{r_1}^+})$ for all $1< r_1<r_2$, where $\Phi$ is the unique solution
 to  problem \eqref{eq_Phi_1}.
\end{Lemma}

\noindent We are now able to prove that the limit $\widetilde C \Phi$ does not depend on the subsequence.

\begin{Lemma}\label{lemma_1}
  Let $\Phi$, $\Psi^+=\frac{\theta_1}{\Upsilon_N}$, and $\Upsilon_N$ be
  as in \eqref{eq_Phi_1}, \eqref{eq:eigenpsi+}, and
  \eqref{eq:upsilonN} respectively. Then, for every $r > 1$, the
  following identity holds true
\begin{multline*} 
  \dfrac{r^N}{r^{N}-1} \left(\dfrac{1}{r}\int_{\SN_+}\Phi(\mathbf
    e_1+r\theta)\Psi^+(\theta)\,d\sigma(\theta) -
    \int_{\SN_+}\Phi(\mathbf
    e_1+\theta)\Psi^+(\theta)\,d\sigma(\theta)\right) \\
  = \Upsilon_N - \int_{\SN_+}\Phi(\mathbf
  e_1+\theta)\Psi^+(\theta)\,d\sigma(\theta).
\end{multline*}
\end{Lemma}
\begin{proof}
For all $r>1$, let us define
\begin{equation*}
  v(r)=\int_{\SN_+}\Phi(\mathbf e_1+r\theta)\Psi^+(\theta)\,d\sigma(\theta).  
\end{equation*}
From \eqref{eq_Phi_1}, $v$ satisfies
\begin{equation}\label{eq_Veronica_semisfere}
 \left(r^{N+1}\left(\dfrac{v}{r}\right)'\right)'=0,\quad\text{in }(1,+\infty),
\end{equation}
hence, by integration, there exists $C\in\R$ such that 
\begin{equation}\label{eq:vdivr}
 \dfrac{v(r)}{r} = v(1)+\dfrac{C}{N}(1-r^{-N}),\quad\text{for all }r\in(1,+\infty).
\end{equation}
From \eqref{eq:Phiinfty} we deduce that $\frac{v(r)}{r}\to 
\int_{\SN_+}\theta_1\Psi^+(\theta)\,d\sigma(\theta)=\frac1{\Upsilon_N}
\int_{\SN_+}\theta_1^2\,d\sigma(\theta)=\Upsilon_N$ as $r\to+\infty$. 
Hence, passing to the limit as $r\to+\infty$  in \eqref{eq:vdivr}, we
obtain
that $\Upsilon_N=v(1)+\frac{C}{N}$,
i.e. $\frac{C}{N}=\Upsilon_N-v(1)$. Then \eqref{eq:vdivr} becomes
\begin{equation*}
 \dfrac{v(r)}{r} = v(1)r^{-N}+\Upsilon_N(1-r^{-N}),\quad\text{for all }r\in(1,+\infty),
\end{equation*}
which directly gives the conclusion.
\end{proof}

\begin{Proposition}\label{p:tildeC}
Let $\{\e_{n}\}_n$, 
 $\{\e_{n_k}\}_k$, and $\widetilde C$ be as in Lemma
 \ref{l:tilde_u_eps_bound}. Then 
 \begin{equation*}
   \widetilde C=\frac{\partial u_0}{\partial x_1}(\mathbf e_1).
 \end{equation*}
\end{Proposition}
\begin{proof}
For all $r\in(\e,3)$ let us define
\begin{equation*}
  \varphi_\e(r)=\int_{\SN_+}u_\e(\mathbf e_1+r\theta)\Psi^+(\theta)\,d\sigma(\theta). 
\end{equation*}
From \eqref{problema} and \eqref{eq:p2} it follows that $\varphi_\e$
satisfies 
\begin{equation*}
\bigg(r^{N+1}\Big(\frac{\varphi_{\e}(r)}{r}\Big)'\bigg)'=0, \quad
\text{in } (\e,3),
\end{equation*}
hence there exists a constant $c_{\e}$ (depending on $\e$ but
independent of $r$) such that 
\begin{equation*}
\Big(\frac{\varphi_{\e}(r)}{r}\Big)'=\frac{c_{\e}}{r^{N+1}}, \quad
\text{in } (\e,3),
\end{equation*}
Integration of  the previous equation in $(R\e,1)$ for a fixed $R\in
(1, 1/\e)$ yields
\begin{equation}\label{eq_prima_lemma_2}
\frac{\varphi_{\e}(R\e)}{R\e} = \varphi_{\e}(1) + \dfrac{c_{\e}}{N} \left(1-(R\e)^{-N}\right).
\end{equation}
On the other hand, integration over $(\e,k\e)$ provides
\begin{equation}\label{eq_seconda_lemma_2}
 \dfrac{\varphi_{\e}(k\e)}{k\e} - \dfrac{\varphi_{\e}(\e)}{\e} =
 \dfrac{c_{\e}}{N\e^N} \left(1-k^{-N}\right),
\quad\text{for all }k\in\left(1,\frac3\e\right).
\end{equation}
From \eqref{eq:84} it follows that
\begin{equation*}
\dfrac{\varphi_{\e}(k\e)}{k\e}=\dfrac{1}{k} \int_{\SN_+}\widetilde
u_{\e}(\mathbf e_1+k\theta)\Psi^+(\theta)
\,d\sigma(\theta)
  \end{equation*}
and \eqref{eq_seconda_lemma_2} becomes
\begin{equation*}
  \dfrac{1}{k} \int_{\SN_+}\widetilde u_{\e}(\mathbf e_1+k\theta)\Psi^+(\theta) \,d\sigma(\theta)
  - \int_{\SN_+}\widetilde u_{\e}(\mathbf e_1+\theta)\Psi^+(\theta) \,d\sigma(\theta) = \dfrac{c_{\e}}{N\e^N} \left(1-k^{-N}\right).
\end{equation*}
Then, from Lemma \ref{l:tilde_u_eps_bound},
\begin{multline*}
  \frac{c_{\e_{n_k}}}{N\e_{n_k}^N} \mathop{\longrightarrow}\limits_{k\to+\infty} \dfrac{\widetilde C k^N}{k^N-1}
  \left(\dfrac{1}{k}\int_{\SN_+}\Phi(\mathbf
    e_1+k\theta)\Psi^+(\theta)\,d\sigma
    - \int_{\SN_+}\Phi(\mathbf e_1+\theta)\Psi^+(\theta) \,d\sigma\right)\\
  = \widetilde C \left(\Upsilon_N - \int_{\SN_+}\Phi(\mathbf
    e_1+\theta)\Psi^+(\theta)\,d\sigma\right),
\end{multline*}
where the last identity is a consequence of Lemma
\ref{lemma_1}. Therefore, passing to the limit along the subsequence
$\e_{n_k}$ in
\eqref{eq_prima_lemma_2} and exploiting 
Lemma \ref{l:tilde_u_eps_bound} and \eqref{convergenza_u_0}, we obtain that 
\begin{multline*}
 \dfrac{\widetilde C}{R} \int_{\SN_+}\Phi(\mathbf e_1+R\theta)\Psi^+(\theta)\,d\sigma
= \int_{\SN_+}u_0(\mathbf e_1+\theta)\Psi^+(\theta)\,d\sigma\\
- \dfrac{\widetilde C}{R^N} \left(\Upsilon_N -
  \int_{\SN_+}\Phi(\mathbf e_1+\theta)\Psi^+(\theta)\,d\sigma\right)
\end{multline*}
for every $R>1$. 
In view of Lemma \ref{lemma_1}, the previous identity 
becomes
\begin{equation}\label{quasi_C_tilde}
 \widetilde C \Upsilon_N = \int_{\SN_+}u_0(\mathbf e_1+\theta)\Psi^+(\theta)\,d\sigma.
\end{equation}
For all $r\in(0,3)$, let us define 
\begin{equation*}
w(r)=\int_{\SN_+}u_0(\mathbf e_1+r\theta)\Psi^+(\theta)\,d\sigma(\theta).  
\end{equation*}
From \eqref{eq:u0}, $w$ satisfies
\begin{equation*}
w''(r)+\frac{N-1}{r}w'(r)-\frac{N-1}{r^2}w(r)=0,\quad\text{in }(0,3),
\end{equation*}
hence, by integration, there exist $c,d\in\R$ such that 
\begin{equation*}
 w(r)=c\,r+d\,r^{1-N},\quad\text{for all }r\in(0,3).
\end{equation*}
The fact that $u_0\in {\mathcal D}^{1,2}(D^+)$ implies that
$d=0$. Hence 
\begin{equation*}
  c=\frac{w(r)}r=\frac1r \int_{\SN_+}u_0(\mathbf
  e_1+r\theta)\Psi^+(\theta)\,d\sigma(\theta),
  \quad\text{for all }r\in(0,3).
\end{equation*}
Moreover 
\begin{equation*}
  \lim_{r\to0^+}\frac{u_0(\mathbf
  e_1+r\theta)}{r}=\nabla u_0(\mathbf
  e_1)\cdot\theta=\frac{\partial u_0}{\partial x_1}(\mathbf
  e_1)\theta_1=\frac{\partial u_0}{\partial x_1}(\mathbf e_1)\Upsilon_N\Psi^+(\theta),
\end{equation*}
thus implying that $c=\frac{\partial u_0}{\partial x_1}(\mathbf
e_1)\Upsilon_N$ and hence 
\begin{equation*}
  \frac1r \int_{\SN_+}u_0(\mathbf
  e_1+r\theta)\Psi^+(\theta)\,d\sigma(\theta)=\frac{\partial u_0}{\partial x_1}(\mathbf
e_1)\Upsilon_N,\quad\text{for all }r\in(0,3).
\end{equation*}
Replacing this last relation into \eqref{quasi_C_tilde} we conclude
that $\widetilde C=\frac{\partial u_0}{\partial x_1}(\mathbf
e_1)$.
\end{proof}

\noindent Combining Lemma \ref{l:tilde_u_eps_bound} and Proposition
\ref{p:tildeC}, we obtain the convergence of $\widetilde u_\e$ to its
limit profile as $\e\to0^+$. 
\begin{Lemma}\label{l:blowright}
 Let $\widetilde u_{\e}$ be defined in \eqref{eq:84}. Then 
\begin{equation*}
 \widetilde u_{\e}\to \frac{\partial u_0}{\partial x_1}(\mathbf
e_1) \Phi, \quad\text{as }\e\to0^+,
\end{equation*}
strongly in $\mathcal
 H_R^+$ for every $R>2$ and  in $C^2(\overline{B_{r_2}^+\setminus
  B_{r_1}^+})$ for all $1< r_1<r_2$, where $\Phi$ is the unique solution
 to  problem \eqref{eq_Phi_1}.
\end{Lemma}

\subsection{Independence in the limit profile at the left}\label{subsec:indipendence}
The second improvement about independence on subsequences concerns
Proposition 6.1 in \cite{FT12} and the convergence of the normalized
eigenfunctions 
\begin{equation}
  \label{eq:151}
  U_\e(x)=\frac{u_\e(x)}{\sqrt{\int_{\Gamma^-_{\tilde k}}u_\e^2d\sigma}}
\end{equation}
to a universal profile (not depending on subsequences),
with $\tilde k$ as in Theorem \ref{t:main} and $\Gamma^-_{\tilde k}$ as in \eqref{eq:defGamma_r-}. We
notice that, for $\e$ small, $U_\e$ solves
\begin{align}\label{eq:24norm}
\begin{cases}
-\Delta U_\e=\lambda_\e p\, U_\e,&\text{in }\Omega^\e,\\
U_\e=0,&\text{on }\partial \Omega^\e,
\end{cases}
\end{align}
and
\begin{equation}
  \label{eq:152}
\int_{\Gamma^-_{\tilde k}}U_\e^2d\sigma=1.
\end{equation}
The following proposition summarizes the results of
\cite[Propositions 6.1 and 6.5]{FT12}.

\begin{Proposition}\label{p:conUeps}{\rm(\cite[Propositions 6.1 and 6.5]{FT12}})
  For every sequence $\e_{n}\to 0^+$ there exist a subsequence
  $\{\e_{n_k}\}_k$, a function $U\in
  C^2(D^-)\cup\big(\bigcup_{t>0}\mathcal H_t^-\big)$, and $\beta<0$  such
  that
\begin{enumerate}[\rm(i)]
\item
$U_{\e_{n_k}}\!\!\!\to U$ strongly in $\mathcal
  H_t^-$ for all $t>0$
and in $C^2(\overline{B_{t_2}^-\!\setminus\!
  B_{t_1}^-})$ for all $0<t_1<t_2$;
\item
  $\int_{\Gamma^-_{\tilde k}}U^2d\sigma=1$;
\item  $U$ solves
\begin{equation}\label{eq:eqU}
\begin{cases}
-\Delta U(x)=\lambda_{k_0}(D^+)p(x)U(x),&\text{in }D^-,\\
U=0,&\text{on }\partial D^-\setminus\{{\mathbf 0}\};
\end{cases}
\end{equation}
\item $\lambda^{N-1} U(\lambda x)
\to\beta\,\frac{x_1}{|x|^N}$ as $\lambda \to 0^+$
strongly in $\mathcal H_t^-$ for every $t>0$ and in
$C^2(\overline{B_{t_2}^-\setminus B_{t_1}^-})$ for all $0<t_1<t_2$.
\end{enumerate}
\end{Proposition}

To prove that the limit profile $U$ in Proposition \ref{p:conUeps}
does not depend on the subsequence, we are going to show that it is
necessarily a multiple of the universal profile $\overline U$ provided
by Lemma \ref{l:barU}; normalization \eqref{eq:152} will  univocally
determine the multiplicative constant. 

A key tool in the proof of Lemma \ref{l:barU}
is the following uniform coercivity type estimate for the quadratic
form associated to equation \eqref{eq:eqU}, whose validity is
strongly related to the nondegeneracy condition~\eqref{eq:54}.
We denote 
\begin{equation}\label{eq:defOmega_r}
\Omega_{r}:=D^-\setminus \overline{B^-_{-r}}\quad \text{for all }r<0.
\end{equation}

\begin{Lemma}\label{lemma0_u}
Let $u\in
  C^2(D^-)\cup\big(\bigcup_{t>0}\mathcal H_t^-\big)$ be a solution to the problem
\begin{equation}\label{problema_u}
\begin{cases}
-\Delta u(x)=\lambda_{k_0}(D^+)p(x)u(x),&\text{in }D^-,\\
u=0,&\text{on }\partial D^-\setminus\{{\mathbf 0}\},
\end{cases}
\end{equation}
where $p\in L^{N/2}(D^-)\setminus\{0\}$ and $\lambda_{k_0}(D^+) \not\in \sigma_p(D^-)$. For any $f\in
L^{N/2}(D^-)$ and $M>0$ there exists $R_{M,f}>0$ such
that, for every $r\in(0,R_{M,f})$,
\begin{equation}
\int_{\Omega_{-r}} \abs{\nabla u(x)}^2\,dx \geq M \int_{\Omega_{-r}}
\abs{f(x)}u ^2 (x)\,dx.
\end{equation}
\end{Lemma}
\begin{pf}
The proof is similar to the proof of Lemma 3.6 in \cite{FT12} and
hence is omitted.
\end{pf}

\noindent We are now in position to prove Lemma \ref{l:barU}.

\begin{pfn}{Lemma \ref{l:barU}}
The existence of a solution to \eqref{problema_U} follows from Proposition
\ref{p:conUeps}. To prove uniqueness, we argue by contradiction and
assume that there exist $U_1,U_1$ solutions to \eqref{problema_U} such
that 
$U_1\neq U_2$. The difference $V=U_1-U_2$  satisfies 
\begin{equation}\label{eq:V}
\begin{cases}
V\in \bigcup_{R>0}C^2(\overline{D^-\setminus
  B_R^-}),
&V\in \bigcup_{R>0}\mathcal H_R^-,\\
-\Delta V=\lambda_{k_0}(D^+) p V, &\text{in
}D^-,\\
V =0, &\text{on }\partial D^-\setminus\{{\mathbf 0}\},\\[5pt]
\lambda^{N-1}V(\lambda \theta)\to 0,&\text{in }C^0({\mathbb S}^{N-1}_-).
\end{cases}
\end{equation}
Let us fix $\delta>0$. From Lemma \ref{lemma0_u}, there exists
$R_\delta>0$ such that
\begin{align}\label{eq:150aft}
&\|2p+x\cdot\nabla p\|_{L^{3N}\big(B^-_{R_\delta}\big)}
\leq 
\bigg(\frac {2N}{\omega_{N-1}}\bigg)^{\!\!\frac5{3N}}
\frac{C_S\delta}{8\lambda_{k_0}(D^+)},\\
\label{eq:76aft}& 
\int_{\Omega_{r}}\Big(|\nabla
      V|^2-\lambda_{k_0}(D^+)p V^2\Big)dx\geq
\frac12\int_{\Omega_{r}} |\nabla V|^2dx,\\
  \label{eq:75aft}& 
\int_{\Omega_{r}}\Big(|\nabla
      V|^2-\lambda_{k_0}(D^+)p V^2\Big)dx \geq\tfrac{4\lambda_{k_0}(D^+)}\delta \!
  \int_{\Omega_{r}} \!\!|2p+x\cdot \nabla p| V^2dx, 
\end{align}
for all $r\in(-R_\delta,0)$, with $C_S$ as in \eqref{eq:sobHt-}.
For all $t>0$, let us define
\begin{align}
&D_V(t)=\frac1{t^{N-2}} \int_{\Omega_{-t}}\Big(|\nabla
  V(x)|^2-\lambda_{k_0}(D^+) p(x) V^2(x)\Big)dx,\label{eq:DV}\\
\label{eq:HV}&H_V(t)=\frac1{t^{N-1}}
\int_{\Gamma_{t}^-}V^2(x)\,d\sigma=
\int_{\SN_-}V^2(t\theta)\,d\sigma(\theta).
\end{align}
Direct calculations (see \cite[Lemma 3.15]{FT12} for details) yield 
\begin{equation*}
  D_V'(t)=-\frac2{t^{N-2}}
\int_{\Gamma_t^-}\bigg|\frac{\partial V}{\partial \nu}\bigg|^2
  d\sigma-\frac{\lambda_{k_0}(D^+)}{t^{N-1}}\int_{\Omega_{-t}}(2p(x)+x\cdot \nabla p(x))
V^2(x)dx.
\end{equation*}
From (\ref{eq:V}), we have that
$$
\int_{\Omega_{-t}}\Big(|\nabla
V(x)|^2-\lambda_{k_0}(D^+) p(x) V^2(x)\Big)dx=-
\int_{\Gamma_t^-} V\,\frac{\partial V}{\partial \nu} \,d\sigma
$$
$\nu=\nu(x)=\frac{x}{|x|}$
which, by Schwarz's inequality, Lemmas \ref{lemma0_u}, and the Poincar\'e type
inequality 
\begin{equation}\label{eq:poincare}
\frac1{t^{N-2}}\int_{\Omega_{-t}}|\nabla
v(x)|^2dx\geq\frac{N-1}{t^{N-1}} \int_{\Gamma_t^-}v^2d\sigma,
\quad\text{for all $t>0$ and $v\in\mathcal H_t^-$}
\end{equation}
proved in \cite[Lemma 3.4]{FT12},
for every $t\in(0,R_\delta)$, up to shrinking   $R_\delta>0$,
 yields 
\begin{align}\label{eq:68aft}
  \int_{\Gamma_t^-} \bigg|\frac{\partial V}{\partial \nu}\bigg|^2
  \!\!d\sigma &\geq \frac{\int_{\Omega_{-t}}\!\!\big(|\nabla
    V|^2\!-\!\lambda_{k_0}(D^+) p
    V^2\big)dx}{\int_{\Gamma_t^-} V^2 \,d\sigma }
 \int_{\Omega_{-t}}\!\!\Big(|\nabla V|^2\!-\!\lambda_{k_0}(D^+)
  p V^2\Big)dx\\
  \notag&\geq\frac{(1-\frac{\delta_0}{2})(N-1)}{t}\int_{\Omega_{-t}}\Big(|\nabla
V|^2-\lambda_{k_0}(D^+)
  p V^2\Big)dx,
\end{align}
with $\delta_0=\frac{2N-5}{4(N-1)}\in(0,1)$.
From Lemma \ref{lemma0_u} and (\ref{eq:68aft}),  up to shrinking $R_\delta$, there holds
\begin{align*}
  -\frac{d}{dt}D_V(t)&\geq \frac{2(1-\delta_0)(N-1)}{t^{N-1}}\int_{\Omega_{-t}}
\Big(|\nabla
  V(x)|^2-\lambda_{k_0}(D^+) p(x) V^2(x)\Big)dx\\
&=\frac{2(1-\delta_0)(N-1)}{t}D_V(t)
\end{align*}
for all $t\in(0,R_\delta)$. Integrating the above inequality, we
obtain that 
\begin{equation}\label{eq:stimaDV}
   D_V(t_1)\geq \bigg(\frac{t_2}{t_1}\bigg)^{\!\!N+\frac12}
  D_V(t_2)\quad \text{for every 
  }t_1,t_2\in(0,R_{\delta_0}) \text{ such that }t_1<t_2.
\end{equation}
 Let us define $\mathcal N_V:(-\infty,0)\to\R$ as 
  \begin{equation*}
    \mathcal N_V(r):=
    \frac{(-r) \int_{\Omega_{r}}\Big(|\nabla
      V(x)|^2-\lambda_{k_0}(D^+)p(x) V^2(x)\Big)dx}{\int_{\Gamma_{-r}^-}V^2(x)\,d\sigma}.
      \end{equation*}
Direct calculations (see \cite[Lemmas 3.15 and 6.2]{FT12} for details
in a similar case) yield 
\begin{equation*}
\dfrac{d}{dr}\mathcal N_V(r) = \nu_1(r) + \nu_2(r),\quad r\in(-\infty,0),
\end{equation*}
where 
\begin{equation*}
\nu_1(r)=-2r\frac{\left(\int_{\Gamma_{-r}^-}\big|\frac{\partial
        V}{\partial \nu}\big|^2
      d\sigma\right)\left(\int_{\Gamma_{-r}^-}V^2(x)\,d\sigma\right)
    -\left(\int_{\Gamma_{-r}^-}V\frac{\partial V}{\partial \nu}
      d\sigma\right)^2}
  {\left(\int_{\Gamma_{-r}^-}V^2(x)\,d\sigma\right)^2}
\geq 0
\end{equation*}
by  Schwarz's inequality and 
\begin{equation*}
\nu_2(r)  = \lambda_{k_0}(D^+)\frac{\int_{\Omega_{r}}(2p(x)+x\cdot \nabla p(x))
    V^2(x)dx}{\int_{\Gamma_{-r}^-}V^2(x) d\sigma}.
\end{equation*}
Hence, for all $r\in(-R_\delta,0)$,
\begin{equation*}
  \frac{\mathcal N_V'(r)}{\mathcal N_V(r)}\geq
-\mathcal I(r)
\end{equation*}
where $\mathcal
I(r)=\frac{\lambda_{k_0}(D^+)}{-r}\big(I(-r)+I\!I(-r)\big)$ with
\begin{align*}
  I(t)=\frac{\int_{\Omega_{-t}\setminus\Omega_{-t^{3/5}}}|2p+x\!\cdot\!
    \nabla p| V^2dx}{\int_{\Omega_{-t}}\!\! \Big(|\nabla
    V|^2\!-\!\lambda_{k_0}(D^+) p V^2\Big)dx},\quad I\!I(t)= \frac{\int_{\Omega_{-t^{3/5}}}|2p+x\!\cdot\! \nabla p|
    V^2dx}{\int_{\Omega_{-t}} \!\!\Big(|\nabla
    V|^2\!-\!\lambda_{k_0}(D^+) p V^2\Big)dx}.
\end{align*}
By H\"older inequality, (\ref{eq:76aft}), \eqref{eq:sobHt-}, and
\eqref{eq:150aft}, $I(t)$ can be estimated as
\begin{align*}
  &I(t)\leq \|2p+x\cdot\nabla p\|_{L^{3N}\big(B^-_{t^{3/5}}\big)}
  \Big|\Omega_{-t}\setminus\Omega_{-t^{3/5}}\Big|^{\frac5{3N}}
  \frac{\Big(\int_{\Omega_{-t}}V^{2^*}dx\Big)^{2/2^*}}
  {\int_{\Omega_{-t}} \Big(|\nabla
    V|^2-\lambda_{k_0}(D^+) p V^2\Big)dx}\\
  &\leq
  \frac{2}{C_S}\bigg(\frac{\omega_{N-1}}{2N}\bigg)^{\!\!\frac5{3N}}
  \|2p+x\cdot\nabla p\|_{L^{3N}\big(B^-_{\breve R_\delta}\big)}t \leq
  \frac{\delta}{4\lambda_{k_0}(D^+)}\,t
\end{align*}
 for all $t\in(0,R_\delta^{5/3})$.
On the other hand, from (\ref{eq:75aft}) and (\ref{eq:stimaDV})
\begin{align*}
  I\!I(t)&= \frac{\int_{\Omega_{-t^{3/5}}}|2p+x\cdot \nabla p|
    V^2dx}{\int_{\Omega_{-t^{3/5}}} \!\!\big(|\nabla
    V|^2\!-\!\lambda_{k_0}(D^+) p V^2\big)dx}
\frac{\int_{\Omega_{-t^{3/5}}} \big(|\nabla
    V|^2-\lambda_{k_0}(D^+) pV^2\big)dx}{\int_{\Omega_{-t}} \big(|\nabla
    V|^2\!-\!\lambda_{k_0}(D^+) pV^2(x)\big)dx}\\
&\leq \frac\delta{4\lambda_{k_0}(D^+)}\, t^{-\frac25(N-2)}\frac{ D_V(t^{3/5})}{
  D_V(t)}\leq\frac\delta{4\lambda_{k_0}(D^+)}\,t
\end{align*}
for all $t\in(0,R_\delta^{5/3})$. Combining the previous estimates we obtain
that $\mathcal I(r)\leq \delta$ and hence 
\begin{equation*}
  \frac{\mathcal N'_V(r)}{\mathcal N_V(r)} \geq -\delta
\end{equation*}
for all $r\in(-R_\delta^{5/3},0)$.

Since $\mathcal N_V(r)\geq 0$ for $r$ close enough to 0 by \eqref{eq:76aft}, we obtain $\mathcal N'_V(r)+\delta \mathcal
N_V(r)\geq 0$, therefore the function $e^{\delta r}\mathcal N_V(r)$
is monotone nondecreasing in a left neighborhood of the origin, so
that in particular $\mathcal N_V(r)$ admits a limit as $r\to 0$.
We observe that Lemma \ref{lemma0_u} and
\eqref{eq:poincare} ensure that $ \lim_{r\to0^-}\mathcal
N_V(r)\geq N-1$. 
 We
claim that 
\begin{equation}\label{eq:claimlim0N-1}
  \lim_{r\to0^-}\mathcal N_V(r)=N-1.
\end{equation}
To prove claim \eqref{eq:claimlim0N-1}, we assume by contradiction
that there exist $\delta>0$ and $\bar r<0$ such that 
$\mathcal N_V(r)\geq N-1+\delta$ for all $r\in(\bar r,0)$. If we integrate the inequality
$$
\dfrac{H'_V(t)}{H_V(t)} = -\dfrac{2}{t}\mathcal
N_V(-t) \leq -\dfrac{2}{t} (N-1+\delta)
$$
over $(t,-\bar r)$, we obtain
\begin{equation}\label{eq:stimasottoHassurdo}
t^{2(N-1+\delta)}H(t) \geq const>0.
\end{equation}
On the other hand, from \eqref{eq:HV} and \eqref{eq:V} it follows that
$H_V(t)=o(t^{2(1-N)})$ as $t\to0^+$, thus contradicting
\eqref{eq:stimasottoHassurdo} and proving claim \eqref{eq:claimlim0N-1}.

From \eqref{eq:claimlim0N-1}, arguing as in \cite[Lemma 6.2, Lemma
6.4, Proposition 6.5]{FT12} one can prove that, if $V\not\equiv0$,
then $V$ would satisfy  
\begin{equation*}
  \lambda^{N-1} V(\lambda \theta)
\to c\,\Psi^-(\theta)\quad\text{as } \lambda \to 0^+\text{ in }C^0(\SN_-)
\end{equation*}
for some $c\neq0$, thus contradicting \eqref{eq:V}.  
 Then $V\equiv 0$.
\end{pfn}

\begin{remark}
 The previous proof does not require assumption
 \eqref{eq:p2}. More generally, the same argument applies replacing
 $p$ with any $L^{N/2}(D^-)$-function and $\lambda_{k_0}(D^+)$ with
 any $\lambda_0\not\in \sigma_p(D^-)$. 
\end{remark}

Combining Proposition \ref{p:conUeps} with Lemma \ref{l:barU}, we can
prove that, due to the universality of the limit profile, the
convergence of 
$U_\e$ does not depend on subsequences.

\begin{Proposition}\label{p:conUeps_nosottosucc}
Let $U_{\e}$ be defined in \eqref{eq:151} and $\overline{U}$ as in
Lemma \ref{l:barU}. Then 
\begin{equation*}
  U_\e\to \frac{\overline{U}}{\sqrt{\int_{\Gamma^-_{\tilde
          k}}\overline{U}^2d\sigma}},\quad\text{as }\e\to0^+,
\end{equation*}
strongly in $\mathcal
  H_t^-$ for every $t>0$
and in $C^2(\overline{B_{t_2}^-\setminus
  B_{t_1}^-})$ for all $0<t_1<t_2$.
\end{Proposition}
\begin{proof}
  Let  $\e_{n}\to 0^+$. From Proposition \ref{p:conUeps}, there exist a subsequence
  $\{\e_{n_k}\}_k$,  $\beta<0$, and a function $U\in
  C^2(D^-)\cup\big(\bigcup_{t>0}\mathcal H_t^-\big)$
  such that $U_{\e_{n_k}}\to U$ strongly in $\mathcal
  H_t^-$ for all $t>0$
and in $C^2(\overline{B_{t_2}^-\!\setminus\!
  B_{t_1}^-})$ for all $0<t_1<t_2$,
$U$ solves \eqref{eq:eqU}, and $\lambda^{N-1} U(\lambda \theta)
\to-\beta\,\Psi^-$ in $C^0(\SN_-)$.
From Lemma \ref{l:barU} it follows that
$\frac{U}{-\beta}=\overline{U}$, whereas part (ii) of Proposition
\ref{p:conUeps} implies that $\beta=-\big(\int_{\Gamma^-_{\tilde
    k}}\overline{U}^2d\sigma\big)^{-1/2}$. Hence the limit $U$ depends
neither on the sequence $\{\e_{n}\}_n$ nor on the subsequence
  $\{\e_{n_k}\}_k$, thus concluding the proof.
\end{proof}

%%%%%%%%%%%%%%%%%%%%%%%%%%%%%%%%%%%%%%%%%%%%%%%%%%%%%%%%%%%%%%%%%%%%%%%%%%%%%%%%%%%%%%%%%%%%%%%%%%%%%%%

\section{Asymptotic behavior of the normalization}\label{sec:asympt-behav-norm}

As already mentioned, in this section the technique is proceeding by
steps. Starting from the right, where we can exploit the strong
convergence \eqref{convergenza_u_0}, we first evaluate the asymptotic
behavior of the denominator of the Almgren quotient at a fixed point
in the corridor, then at $\e$-distance from the left junction in the
corridor, and finally at a fixed distance from the left junction in
$D^-$.

Following \cite{FT12}, for every $r\in(0,1)$ and $t>\e$ we define
\begin{align}
  \label{Htilde} & \widetilde H_\e (r) := \int_{\Sigma}
  u_\e^2(r,\e x') dx' = \e^{1-N} \int_{\Sigma_\e} u_\e^2(r,x')
  dx'
  = \e^{1-N} H_\e^c(r),\\
 \label{Hmeno} & H_\e^-(t):= \dfrac{1}{t^{N-1}} \int_{\Gamma_t^-} u_\e^2 d\sigma,
\end{align}
where $\Gamma_t^-$ is defined in \eqref{eq:defGamma_r-} and
$\Sigma_\e:=\{x'\in \R^{N-1}:|x'|<\e\}$.
We also define, for
every $r>0$,
\begin{equation*}
\widehat\Omega_r=D^-\cup \{(x_1,x')\in T_1^+:x_1<r\}
\end{equation*}
and $\mathcal H_r$ \label{pag:defHr}as the completion of
\[
\mathcal D_r:=\big\{v\in C^\infty(\overline{\widehat\Omega_r}):
\mathop{\rm supp}v\Subset \widehat D\big\}
\]
 with respect to the norm $\big(\int_{\widehat\Omega_r}|\nabla
v|^2dx\big)^{\!1/2}$, 
 i.e.
$\mathcal H_r$ is the space of functions with finite energy in
$\widehat\Omega_r$ vanishing on $\{(x_1,x')\in \partial \widehat\Omega_r:
x_1< r\}$.

\begin{Lemma}\label{lemma_blow_up_canale}
  Let us fix $x_0\in (0,1)$ and define
\begin{equation}\label{def_blow_up_canale}
w_\e:\widetilde\Omega_{x_0,\e}\to\R,\quad 
 w_\e(x_1,x')=\dfrac{u_\e\big(\e(x_1-1)+x_0, \e x'\big)}{\big(\widetilde H_\e (x_0)\big)^{1/2}}
\end{equation}
where 
\begin{align*}
  \widetilde\Omega_{x_0,\e}:=\bigg\{&
(x_1,x')\in \R\times\R^{N-1}:\ x_1< 1-\frac{x_0}{\e}\bigg\} \\
&\cup
\bigg\{(x_1,x')\in \R\times\R^{N-1}:\  1-\frac{x_0}{\e}\leq x_1 \leq 1+
\frac{1-x_0}{\e},\ |x'|<1 \bigg\}\\
& \cup
\bigg\{(x_1,x')\in \R\times\R^{N-1}:\ x_1> 1+\frac{1-x_0}{\e}\bigg\}.
\end{align*}
Then
\begin{equation*}
w_\e(x_1,x')\to e^{\sqrt{\lambda_1(\Sigma)}(x_1-1)}\psi_1^\Sigma(x') \quad
  \text{in } C^2_{\rm loc}(\overline{T_1})
\text{ and in }\mathcal
H_r\text{ for every }r\in\R,
\end{equation*}
as $\e\to0^+$,
where $\lambda_1(\Sigma)$ the first
eigenvalue of $-\Delta_{x'}$ on $\Sigma$ under null Dirichlet boundary
conditions, $\psi_1^\Sigma$ is the corresponding positive
$L^2(\Sigma)$-normalized eigenfunction (see \eqref{eq:eigensigma}), and
$T_1$ is defined in~\eqref{eq:T1}.
\end{Lemma}
\begin{pf}
For every $r\in \R$, we define 
\begin{equation*}
  \widetilde\Omega^{x_0,\e}_r:=\{(x_1,x')\in
\widetilde\Omega_{x_0,\e}:\ x_1<r\}.
\end{equation*}
Let us fix $r>1$. Then, for $\e$ sufficiently small, $\e(r-1)+x_0\in(0,1)$.
By direct computations we have that 
\begin{equation}\label{eq:sti1}
  \frac{\int_{\widetilde\Omega^{x_0,\e}_r}\big(|\nabla
    w_\e(x)|^2-\e^2\lambda_\e p(\e(x-{\mathbf e}_1)+x_0 {\mathbf
      e}_1)w_\e^2(x)\big)\,dx}
{\int_{\Sigma}w_\e^2(r,x')\,dx'}=\mathcal N_\e(\e(r-1)+x_0)
\end{equation}
where, for all $t\in(0,1)$, 
\begin{align*}
  \mathcal N_\e(t)=\frac{\e
\int_{\{(x_1,x')\in\Omega^\e:x_1<t\}}\Big(|\nabla
  u_\e(x)|^2-\lambda_\e p(x) u_\e^2(x)\Big)dx
}{ H_\e^c(t)}.
\end{align*}
From \cite[Lemma 3.21, Lemma 4.5, and Corollary 2.6 ]{FT12} it follows
that for every $\delta>0$ there exists $\e_{\delta,r,x_0}>0$
(depending on $\delta$, $r$, and $x_0$) such that
\begin{equation}\label{eq:sti2}
  \mathcal N_\e(\e(r-1)+x_0)\leq
  (1+\delta)\sqrt{\lambda_1(\Sigma)}\quad\text{for all }\e\in(0,\e_{\delta,r,x_0}).
\end{equation}
Furthermore, \cite[Lemma
3.6]{FT12} implies that, up to shrinking $\e_{\delta,r,x_0}>0$, for
all $\e\in(0,\e_{\delta,r,x_0})$,
\begin{multline}\label{eq:sti3}
  \lambda_\e\e^2\int_{\widetilde\Omega^{x_0,\e}_r} p(\e(x-{\mathbf e}_1)+x_0
  {\mathbf e}_1)|w_\e(x)|^2\,dx=\frac{\e^{2-N}\lambda_\e}{\widetilde H_\e
    (x_0)}\int_{\Omega^\e_{\e(r-1)+x_0}} p(x)
  u_\e^2(x)dx\\
  \leq \delta \frac{\e^{2-N}}{\widetilde H_\e
    (x_0)}\int_{\Omega^\e_{\e(r-1)+x_0}} |\nabla
  u_\e(x)|^2dx=\delta\int_{\widetilde\Omega^{x_0,\e}_r}|\nabla w_\e(x)|^2\,dx,
\end{multline}
where
$\Omega^\e_{\e(r-1)+x_0}=\{(x_1,x')\in\Omega^\e:x_1<\e(r-1)+x_0\}$.
Collecting \eqref{eq:sti1}, \eqref{eq:sti2}, and \eqref{eq:sti3}, and
recalling \eqref{Htilde}, we obtain that
\begin{equation}\label{eq:sti4}
  \int_{\widetilde\Omega^{x_0,\e}_r}|\nabla w_\e(x)|^2\,dx\leq 
  \frac{(1+\delta) \sqrt{\lambda_1(\Sigma)}}{1-\delta}\,
  \frac{\widetilde H_\e(\e(r-1)+x_0) }{\widetilde H_\e(x_0) }
\end{equation}
for all $\e\in(0,\e_{\delta,r,x_0})$. 
From \cite[Lemma 3.20]{FT12} it follows that 
\begin{equation*}
  \frac{\widetilde H_\e'(t)}{\widetilde H_\e(t)}=\frac2\e\mathcal
  N_\e(t),\quad\text{for all }t\in(0,1).
\end{equation*}
Integrating the above identity and using again Lemma 3.21 of
\cite{FT12} and \eqref{eq:sti2}, it follows that,  up to shrinking $\e_{\delta,r,x_0}>0$, for
all $\e\in(0,\e_{\delta,r,x_0})$,
\begin{equation}\label{eq:sti5}
   \frac{\widetilde H_\e(\e(r-1)+x_0) }{\widetilde H_\e(x_0) }\leq e^{
2(r-1)e^\delta(1+\delta)  \sqrt{\lambda_1(\Sigma)}
}.
\end{equation}
In view of (\ref{eq:sti4}) and (\ref{eq:sti5}), we have proved that for
every $r>1$ there exists $\e_{r,x_0}>0$ such that
\begin{equation}\label{eq:sti6}
\{w_{\e}\}_{\e\in (0,\e_{r,x_0})} \text{ is bounded in }\mathcal H_r.
\end{equation}
Let $\e_n\to 0^+$.
From (\ref{eq:sti6}) and a diagonal
process, we deduce that  there exist a
subsequence $\e_{n_k}\to 0^+$  and some $w\in
\bigcup_{r>1}\mathcal H_r$ such that $w_{\e_{n_k}}\rightharpoonup w$ weakly in $\mathcal H_r$ for every
$r>1$. In particular $w_{\e_{n_k}}\to w$ a.e., so
that $w\equiv 0$ in $\R^N\setminus T_1$.  Passing to the weak limit in
\begin{equation}\label{problema_weps}
\begin{cases}
-\Delta w_\e=\e^2 \lambda_\e p(\e(x-{\mathbf e}_1)+x_0
  {\mathbf e}_1) w_\e,&\text{in }\widetilde\Omega_{x_0,\e},\\
w_\e=0,&\text{on }\partial\widetilde\Omega_{x_0,\e},
\end{cases}
\end{equation}
along the subsequence $\e_{n_k}$ we obtain that $w$ satisfies 
\begin{equation}\label{problema_wlim}
\begin{cases}
-\Delta w=0,&\text{in }T_1,\\
w=0,&\text{on }\partial T_1.
\end{cases}
\end{equation}
By classical elliptic estimates, we also have that $w_{\e_{n_k}}\to
w$ in $C^2_{\rm loc }(\overline{T_1})$. Therefore, multiplying
\eqref{problema_wlim} by $w$ and integrating in  $T_{1,r}$ where 
$T_{1,r}:=\{(x_1,x')\in T_1:\, x_1< r\}$, we obtain
\begin{align}\label{eq:conve1}
\int_{\Sigma}\frac{\partial w_{\e_{n_k}}}{\partial x_1}(r,x')w_{\e_{n_k}}(r,x')\,dx'
\to&
\int_{\Sigma}\frac{\partial w}{\partial x_1}(r,x')w(r,x')\,dx'\\
&\notag=\int_{T_{1,r} }|\nabla w(x)|^2dx\quad\text{as }k\to+\infty .
\end{align}
On the other hand,  multiplication of \eqref{problema_weps} by $w_{\e_{n_k}}$ and integration by parts over
$\widetilde{\Omega}_{r}^{x_0,\e_{n_k}}$ yield
\begin{align}\label{eq:conve2}
  \int_{\widetilde{\Omega}_{r}^{x_0,\e_{n_k}}}|\nabla
  w_{\e_{n_k}}(x)|^2dx&=
\int_{\Sigma}\frac{\partial w_{\e_{n_k}}}{\partial x_1}(r,x')w_{\e_{n_k}}(r,x')\,dx'\\
&\notag + \lambda_{\e_{n_k}}\e_{n_k}^2
  \int_{\widetilde{\Omega}_{r}^{x_0,\e_{n_k}}}
 p(\e(x-{\mathbf e}_1)+x_0
  {\mathbf e}_1) |w_{\e_{n_k}}(x)|^2\,dx.
\end{align}
From \eqref{eq:sti3} it follows that 
\begin{equation}\label{eq:conve3}
\e_{n_k}^2
  \int_{\widetilde{\Omega}_{r}^{x_0,\e_{n_k}}}
 p(\e(x-{\mathbf e}_1)+x_0
  {\mathbf e}_1) |w_{\e_{n_k}}(x)|^2\,dx\to0\quad \text{as }k\to+\infty,
\end{equation}
which, in view of \eqref{eq:conve1} and \eqref{eq:conve2}, implies that
$\|w_{\e_{n_k}}\|_{\mathcal H_r}\to
\|w\|_{\mathcal H_r}$ and
then $w_{\e_{n_k}}$ converges to $w$ strongly in $\mathcal
H_r$ for every $r>1$. Then, from \eqref{eq:sti1}, \eqref{eq:conve3},
and \eqref{eq:sti2} we deduce that, for all $r>1$ and $\delta>0$, 
\begin{equation*}
  \frac{\int_{T_{1,r}}|\nabla w(x)|^2\,dx}
{\int_{\Sigma}w^2(r,x')\,dx'}\leq  (1+\delta)\sqrt{\lambda_1(\Sigma)}, 
\end{equation*}
which implies that, for all $r>1$, 
\begin{equation*}
  \frac{\int_{T_{1,r}}|\nabla w(x)|^2\,dx}
{\int_{\Sigma}w^2(r,x')\,dx'}\leq  \sqrt{\lambda_1(\Sigma)}. 
\end{equation*}
Hence, from \cite[Lemma 2.5]{FT12} it follows that 
$w(x_1,x')=C\,e^{\sqrt{\lambda_1(\Sigma)}(x_1-1)}\psi_1^\Sigma(x')$
for some constant $C\neq 0$.
Thanks to the definition of $w_\e$, we have that
\begin{equation*}
\int_{\Sigma}{w_\e}^{\!2}(1,x')dx' =1, 
\end{equation*}
and then $C^2=1$.

It remains to prove that $C>0$ so that $C=1$. We assume by contradiction that  $C<0$.
By the convergence $w_{\e_{n_k}}(1,x') \to w(1,x')$ in $C^2(\Sigma)$, it
follows that, if $k$ is sufficiently large, then  $u_{\e_{n_k}}(x_0,x')<0$ for every $x'\in\R^{N-1}$ such that
$|x'|<\e$.
 
From \cite[Corollary 1.3]{FT12}, for every $r$ small, there exists
$\e_r>0$ such that $u_\e>0$ on $\Gamma_r^-$ for all
$\e\in(0,\e_r)$. Hence there exists a sub-subsequence
$\{\e_{n_{k_j}}\}_j$  such that $u_{\e_{n_{k_j}}}>0$ on $\Gamma^-_{1/j}$.
 Therefore, for $j$ large, the functions
\begin{equation*}
  v_j=\begin{cases}
    u_{\e_{n_{k_j}}},&\text{in }D^-\setminus B_{1/j}^-,\\
    u ^+_{\e_{n_{k_j}}},&\text{in } B_{1/j}^-\cup\{(x_1,x')\in \mathcal C_\e:\, x_1\leq
    x_0
    \},\\
    0,&\text{in } \{(x_1,x')\in \Omega^{\e_{n_{k_j}}}:\, x_1>x_0
    \},
  \end{cases}
\quad
\widetilde v_j := \dfrac{v_{j}}{\big(\int_{\Omega^{\e_{n_{k_j}}}}p\,v_{j}^2dx\big)^{\!\frac12}}
\end{equation*}
are well-defined and belong to $\Di$ (if trivially extended to the whole $\R^N$).

Let $A_j=D^-\setminus B_{1/j}^-$. Then $\widetilde v_j$
satisfies
\begin{equation*}
\begin{cases}
-\Delta \widetilde v_j=\lambda_{\e_{n_{k_j}}} p \widetilde v_j,&\text{in }A_j,\\
\widetilde v_j=0,&\text{on }\partial A_j\cap\partial D^-,\\
\int_{\R^N}p\,\widetilde v_j^2\,dx=1.
\end{cases}
\end{equation*}
Testing equation \eqref{problema} for $\e=\e_{n_{k_j}}$ with $v_j$ we
obtain 
\[
\int_{\{(x_1,x')\in \Omega^{\e_{n_{k_j}}}:\, x_1\leq
    x_0
    \}}|\nabla v_j|^2dx= 
\lambda_{\e_{n_{k_j}}}\int_{\{(x_1,x')\in \Omega^{\e_{n_{k_j}}}:\, x_1\leq
    x_0
    \}}pv_j^2dx,
\] 
hence 
\begin{align*}
&\int_{\R^N}p \widetilde v_j^2dx=\int_{\{(x_1,x')\in \Omega^{\e_{n_{k_j}}}:\, x_1\leq
    x_0
    \}}p \widetilde v_j^2dx=1,\\
&\int_{\R^N}|\nabla \widetilde v_j|^2dx=\int_{\{(x_1,x')\in \Omega^{\e_{n_{k_j}}}:\, x_1\leq
    x_0
    \}}|\nabla \widetilde v_j|^2dx=
\lambda_{\e_{n_{k_j}}}.
\end{align*}
Hence $\{\widetilde v_j\}_{j}$ is bounded in
$\Di$ and, along a subsequence, $\Di$-weakly converges to some
$\widetilde v\in\Di$ such that $\int_{\R^n}p\widetilde
v^2\,dx=1$, $\mathop{\rm supp}\widetilde v\subset D^-$, and 
\begin{equation*}
\begin{cases}
-\Delta \widetilde v=\lambda_{k_0}(D^+) p \widetilde v,&\text{in }D^-,\\
\widetilde v =0,&\text{on }\partial D^-,
\end{cases}
\end{equation*}
thus implying that $\lambda_{k_0}(D^+)\in \sigma_p(D^-)$, in
contradiction with assumption \eqref{eq:54}.

We have then proved that  the limit $w$ depends
neither on the sequence $\{\e_{n}\}_n$ nor on the subsequence
  $\{\e_{n_k}\}_k$, thus concluding the proof.
\end{pf}

\noindent Let us define 
\begin{equation}\label{eq:phieps}
  \phi_\e:(0,1)\to\R,\quad \phi_\e(t)=\int_{\Sigma}u_\e(t,\e x')\psi_1^\Sigma(x')dx'.
\end{equation}
As a consequence of Lemma \ref{lemma_blow_up_canale}, the following
result holds.
\begin{Corollary}\label{c:hphi}
For every $x_0\in (0,1)$
\begin{equation*}
  \lim_{\e\to0^+}\frac{\phi_\e(x_0)}{\sqrt{\widetilde H_\e(x_0)}}=1,
\end{equation*}
where $\phi_\e$ is defined in \eqref{eq:phieps} and $\widetilde H_\e$
in \eqref{Htilde}.
\end{Corollary}
\begin{pf}
 From Lemma \ref{lemma_blow_up_canale} it follows that 
$w_{\e}(1,x') \to \psi_1^\Sigma (x')$ in $C^2(\Sigma)$ as $\e\to0^+$, i.e.
\begin{equation*}
\dfrac{u_\e(x_0,\e x')}{\sqrt{\widetilde H_\e(x_0)}} = w_\e(1,x') \to
\psi_1^\Sigma(x'), \text{ in $C^2(\Sigma)$},
\end{equation*}
which easily implies the conclusion.
\end{pf}

\begin{Proposition}\label{asintotico_H_epsilon_canale}
 For every $x_0\in(0,1)$ 
 \begin{equation*}
  \lim_{\e\to0^+} \e^{-1} e^{-\frac{\sqrt{\lambda_1(\Sigma)}}{\e}(x_0-1)}\sqrt{\widetilde H_\e(x_0)}
= \dfrac{\partial u_0}{\partial x_1}({\mathbf e}_1) \, \int_{\Sigma}\Phi(1,x')\psi_1^\Sigma(x')dx', 
 \end{equation*}
being $u_0$ as in \eqref{eq:u0} and  $\Phi$ the unique solution to problem \eqref{eq_Phi_1}.
\end{Proposition}
\begin{pf}
 By virtue of Corollary \ref{c:hphi}, it is sufficient to prove that
\begin{equation}\label{eq:asphieps}
\lim_{\e\to0^+} \e^{-1} e^{-\frac{\sqrt{\lambda_1(\Sigma)}}{\e}(x_0-1)}\phi_\e(x_0)
= \dfrac{\partial u_0}{\partial x_1}({\mathbf e}_1) \,
\int_{\Sigma}\Phi(1,x')\psi_1^\Sigma(x')dx'.
\end{equation}
Recalling that, from \eqref{problema} and \eqref{eq:p2}, $u_\e$ is
harmonic in the corridor, $\phi_\e$ satisfies 
\begin{equation*}
(\phi_\e)''(t)=\frac{\lambda_1(\Sigma)}{\e^2}\phi_\e(t),\quad\text{for
all }t\in(0,1),
\end{equation*}
which can be rewritten as
\begin{equation*}
 \left(e^{\frac{2}{\e}\sqrt{\lambda_1(\Sigma)}(t-1)}\bigg(e^{-\frac{\sqrt{\lambda_1(\Sigma)}}{\e}(t-1)}\phi_\e(t)\bigg)'\right)'=0
 \quad\text{in }(0,1). 
\end{equation*}
Hence 
\begin{align}
  \label{eq_Veronica_phi} &
  \bigg(e^{-\frac{\sqrt{\lambda_1(\Sigma)}}{\e}(t-1)}\phi_\e(t)\bigg)'
  =C_\e e^{-\frac{2}{\e}\sqrt{\lambda_1(\Sigma)}(t-1)},\quad\text{for
    all }t\in(0,1) \\
  \label{eq_Laura_phi} & \phi_\e(t) = A_\e
  e^{\frac{\sqrt{\lambda_1(\Sigma)}}{\e}(t-1)} + B_\e
  e^{-\frac{\sqrt{\lambda_1(\Sigma)}}{\e}(t-1)},\quad\text{for
    all }t\in(0,1) ,
\end{align}
being $C_\e$, $A_\e$ and $B_\e$ some real constants depending on
$\e$ (and independent of $t$). The proof of  the proposition is
divided in several steps.

\smallskip\noindent
{\bf Step 1:} we claim that 
\begin{equation}\label{eq:CepsBeps}
C_\e = -2\sqrt{\lambda_1(\Sigma)}\,
   \frac{B_\e}{\e}.
 \end{equation}
 Indeed, fixing $h>1$ and integrating
   \eqref{eq_Veronica_phi} over $(1-h\e,1-\e)$, we obtain that
\begin{align*}
 \dfrac{C_\e \e}{2\sqrt{\lambda_1(\Sigma)}}\left(e^{2h\sqrt{\lambda_1(\Sigma)}}-e^{2\sqrt{\lambda_1(\Sigma)}}\right)
&= \int_{1-h\e}^{1-\e} \bigg(e^{-\frac{\sqrt{\lambda_1(\Sigma)}}{\e}(t-1)}\phi_\e(t)\bigg)'dt \\
&= \int_{1-h\e}^{1-\e} \left(A_\e + B_\e e^{-2\frac{\sqrt{\lambda_1(\Sigma)}}{\e}(t-1)}\right)'dt\\
&= B_\e \left(e^{2\sqrt{\lambda_1(\Sigma)}} - e^{2h\sqrt{\lambda_1(\Sigma)}}\right).
\end{align*}

\smallskip\noindent
{\bf Step 2:} we claim that, for every $R,h>0$,
\begin{equation*}
  e^{\sqrt{\lambda_1(\Sigma)}R}\int_{\Sigma}\Phi(1-R,x')\psi_1^\Sigma(x')\,dx' -
  e^{\sqrt{\lambda_1(\Sigma)}h}
  \int_{\Sigma}\Phi(1-h,x')\psi_1^\Sigma(x')\,dx' =0.
\end{equation*}
Let us define 
\begin{equation*}
  \phi:(-\infty,1]\to\R,\quad \phi(t)=\int_{\Sigma}\Phi(t, x')\psi_1^\Sigma(x')dx'.
\end{equation*}
Since $\Phi$ is harmonic on its domain,
$\phi$ solves
\begin{equation*}
\bigg(e^{-\sqrt{\lambda_1(\Sigma)}(t-1)}\phi(t)\bigg)'
  =C e^{-2\sqrt{\lambda_1(\Sigma)}(t-1)},\quad\text{for
    all }t\in(-\infty,1],
\end{equation*}
being $C$ a real constant (independent of $t$). 
Integrating the above equation over $(1-\rho ,1)$, we obtain
\[
\phi(1) - e^{\rho \sqrt{\lambda_1(\Sigma)}}\phi(1-\rho ) =
\dfrac{C}{2\sqrt{\lambda_1(\Sigma)}}\left(e^{2\rho \sqrt{\lambda_1(\Sigma)}}-1\right)\
\text{for
all }\rho >0
\]
and then
\begin{equation}\label{eq:Cj}
C= \dfrac{2\sqrt{\lambda_1(\Sigma)}}{e^{2\rho \sqrt{\lambda_1(\Sigma)}}-1}
\left(\phi(1) - e^{\rho \sqrt{\lambda_1(\Sigma)}}\phi(1-\rho )\right) \quad\text{for
all }\rho >0.
\end{equation}
From \cite[Lemma 2.9 (ii)]{FT12},
$\Phi(x_1,x')=O(e^{\sqrt{\lambda_1(\Sigma)}\frac{x_1-1}{2}})$ as
$x_1\to-\infty$. Hence
 \[
\phi(1-\rho )=\int_\Sigma \Phi(1-\rho ,x')\psi_1^\Sigma(x')\,dx'
=O(e^{-\frac{\rho }{2}\sqrt{\lambda_1(\Sigma)}})
\]
 as
$\rho \to+\infty$. Therefore, letting $\rho \to+\infty$ in \eqref{eq:Cj}
implies that $C=0$.
This yields that 
\begin{equation}\label{lemma2_Phi_1}
e^{\rho \sqrt{\lambda_1(\Sigma)}}\phi(1-\rho )=\phi(1)
\end{equation}  for any $\rho >0$, thus proving the claim.
 
\smallskip\noindent {\bf Step 3:} we claim that $C_\e \to 0$ as $\e
\to 0$. Indeed, fixing $h>R>0$ and  integrating \eqref{eq_Veronica_phi} over
$(1-h\e,1-R\e)$, we obtain that
\[
 e^{R\sqrt{\lambda_1(\Sigma)}}\dfrac{\phi_\e(1-R\e)}{\e} - e^{h\sqrt{\lambda_1(\Sigma)}}\dfrac{\phi_\e(1-h\e)}{\e}
=
\dfrac{C_\e}{2\sqrt{\lambda_1(\Sigma)}}\left(e^{2h\sqrt{\lambda_1(\Sigma)}}-e^{2R\sqrt{\lambda_1(\Sigma)}}\right),
\]
from which, thank to Lemma \ref{l:blowright} and Step 2, it follows that  
\begin{multline*}
  C_\e =
  \dfrac{2\sqrt{\lambda_1(\Sigma)}}{e^{2h\sqrt{\lambda_1(\Sigma)}}-e^{2R\sqrt{\lambda_1(\Sigma)}}}
  \left(e^{R\sqrt{\lambda_1(\Sigma)}}\dfrac{\phi_\e(1-R\e)}{\e} -
    e^{h\sqrt{\lambda_1(\Sigma)}}\dfrac{\phi_\e(1-h\e)}{\e}\right)\\
  \stackrel{\e\to0^+}{\longrightarrow}
  \dfrac{2\sqrt{\lambda_1(\Sigma)}\frac{\partial u_0}{\partial
      x_1}(\mathbf
    e_1)}{e^{2h\sqrt{\lambda_1(\Sigma)}}-e^{2R\sqrt{\lambda_1(\Sigma)}}}
  \bigg(e^{R\sqrt{\lambda_1(\Sigma)}}\int_\Sigma
  \Phi(1-R,x')\psi_1^\Sigma(x')\,dx'\\
  -e^{h\sqrt{\lambda_1(\Sigma)}}\int_\Sigma
  \Phi(1-h,x')\psi_1^\Sigma(x')\,dx'\bigg) =0
\end{multline*}
thanks to the previous step.
 
\smallskip\noindent
{\bf Step 4:} we claim that
\begin{equation}\label{limite_c_eps^1}
\dfrac{A_\e}{\e}\to \dfrac{\partial u_0}{\partial x_1}(\mathbf e_1)\int_\Sigma \Phi(1,x')\psi_1^\Sigma(x')\,dx'
\quad \text{as } \e\to 0^+.
\end{equation}
For all $R>0$, the convergence
\[
\dfrac{\phi_\e(1-R\e)}{\e} \to \dfrac{\partial u_0}{\partial
  x_1}(\mathbf e_1)\int_\Sigma \Phi(1-R,x')
\psi_1^\Sigma(x')\,dx',
\quad \text{as }\e\to 0^+,
\]
which is a consequence of Lemma \ref{l:blowright}, implies, in view
of \eqref{eq_Laura_phi}, that
\[
\dfrac{A_\e}{\e}e^{-R\sqrt{\lambda_1(\Sigma)}} + \dfrac{B_\e}{\e}e^{R\sqrt{\lambda_1(\Sigma)}}
\to \dfrac{\partial u_0}{\partial x_1}(\mathbf e_1)\int_\Sigma \Phi(1-R,x')\psi_1^\Sigma(x')\,dx'
 \text{ as }\e\to 0^+.
\]
Claim \eqref{limite_c_eps^1} follows from the above convergence, Steps 1 and 3, which imply that 
\[
\frac{B_\e}{\e}=-\frac{C_\e}{2\sqrt{\lambda_1(\Sigma)}} = o(1)
\] as
$\e\to0^+$, and \eqref{lemma2_Phi_1}.
 
\smallskip\noindent {\bf Step 5:} we claim that, for every
$x_0\in(0,1)$,
\begin{equation}\label{ultimo_passo_asintotico_corridoio}
  {B_\e}\e^{-1} e^{-2\frac{\sqrt{\lambda_1(\Sigma)}}{\e}(x_0-1)} \to 0 \quad \text{as }\e\to 0^+.
\end{equation}
Let us fix $\overline x\neq 1$. Then, taking into account
\eqref{eq_Laura_phi}, Lemma \ref{lemma_blow_up_canale} and Corollary
\ref{c:hphi}  imply that
\begin{align}\label{equazione_ultimo_passo_asintotico_corridoio}
  e^{\sqrt{\lambda_1(\Sigma)}(\overline x-1)} &=\lim_{\e\to0^+} \dfrac{\phi_\e(\e(\overline x-1)+x_0)}{\phi_\e(x_0)}\\
  &\notag= \lim_{\e\to0^+}\dfrac{\frac{A_\e}{B_\e}
    e^{2\frac{\sqrt{\lambda_1(\Sigma)}}{\e}(x_0-1)}e^{\sqrt{\lambda_1(\Sigma)}(\overline
      x-1)} + e^{-\sqrt{\lambda_1(\Sigma)}(\overline x-1)}
  }{\frac{A_\e}{B_\e}
    e^{2\frac{\sqrt{\lambda_1(\Sigma)}}{\e}(x_0-1)} +1 }.
\end{align}
By contradiction, let us assume that
\eqref{ultimo_passo_asintotico_corridoio} is not true and hence that
there exist $\alpha>0$ and a sequence $\{\e_{n}\}_n\to 0^+$ such that
\[
\abs{{B_{\e_n}}\e_n^{-1}
  e^{-2\frac{\sqrt{\lambda_1(\Sigma)}}{\e_n}(x_0-1)}} \geq
\alpha>0,\quad\text{for all }n,
\]
which, in view of \eqref{limite_c_eps^1}, implies that 
\[
\dfrac{A_{\e_n}}{B_{\e_n}}
  e^{2\frac{\sqrt{\lambda_1(\Sigma)}}{\e_n}(x_0-1)}=O(1)\quad\text{as }n\to+\infty.
\]
Then there exist $\ell\in\R$ and a subsequence $\{\e_{n_k}\}_{k}$ such
that 
\[
\dfrac{A_{\e_{n_k}}}{B_{\e_{n_k}}}
e^{2\frac{\sqrt{\lambda_1(\Sigma)}}{{\e_{n_k}}}(x_0-1)} \to \ell.
\]
If $\ell \neq -1$, then from
\eqref{equazione_ultimo_passo_asintotico_corridoio} 
it follows that 
\[
  e^{\sqrt{\lambda_1(\Sigma)}(\overline x-1)}=\frac{\ell\, e^{\sqrt{\lambda_1(\Sigma)}(\overline
      x-1)} + e^{-\sqrt{\lambda_1(\Sigma)}(\overline x-1)}
  }{\ell +1 },
\]
thus contradicting the fact that $\bar x\neq1$. On the other hand, if
$\ell=-1$ the limit at the second line of
\eqref{equazione_ultimo_passo_asintotico_corridoio}
is  $\pm \infty$, giving again rise to a contradiction.

\smallskip\noindent We are now in position to conclude the proof.
From \eqref{eq_Laura_phi}, \eqref{limite_c_eps^1}, and
\eqref{ultimo_passo_asintotico_corridoio} it follows that
\begin{align*}
 \lim_{\e\to 0^+} &\e^{-1} e^{-\frac{\sqrt{\lambda_1(\Sigma)}}{\e}(x_0-1)} \phi_\e(x_0)\\
&= \lim_{\e\to 0^+} \e^{-1} e^{-\frac{\sqrt{\lambda_1(\Sigma)}}{\e}(x_0-1)}
\left(A_\e e^{\frac{\sqrt{\lambda_1(\Sigma)}}{\e}(x_0-1)} + B_\e e^{-\frac{\sqrt{\lambda_1(\Sigma)}}{\e}(x_0-1)} \right)\\
&= \lim_{\e\to 0^+} \left\{\dfrac{A_\e}{\e} + \dfrac{B_\e}{\e}e^{-2\frac{\sqrt{\lambda_1(\Sigma)}}{\e}(x_0-1)} \right\}
= \dfrac{\partial u_0}{\partial x_1}(\mathbf e_1)\int_\Sigma \Phi(1,x')\psi_1^\Sigma(x')\,dx',
\end{align*}
thus completing the proof.
\end{pf}

\bigskip

In order to come to a further step in our analysis, we find useful to recall some basic facts in \cite{FT12}
concerning the blow-up limit at the left junction.
We define
\begin{align}\label{eq:121}
\widehat{u}_\e:\widehat \Omega^\e\to\R,\quad
\widehat u_\e(x)=\frac{u_\e(\e x)}{\sqrt{\int_\Sigma u_\e^2(\e,\e x')\,dx'}},
\end{align}
where
\begin{equation*}
\widehat \Omega^\e:
=D^-\cup\{(x_1,x')\in T_1:0\leq x_1\leq 1/\e\}\cup
\{(x_1,x'):x_1> 1/\e\}.
\end{equation*}
We observe that $\widehat u_\e$ solves
\begin{equation*}
\begin{cases}
-\Delta \widehat u_\e(x)=\e^2\lambda_\e p(\e x) \widehat u_\e(x),&\text{in }\widehat \Omega^\e,\\
\widehat u_\e=0,&\text{on }\partial\widehat \Omega^\e.
\end{cases}
\end{equation*}
We let $\widehat D$ as in \eqref{eq:hatD}, and consider, for all
$r>0$, $\mathcal H_r$ as defined at page \pageref{pag:defHr}. 
The change of variable $y'=\e x'$ yields
\begin{equation}\label{eq:normalization_eqhat}
\int_{\Sigma}\widehat u_\e^2(1,x')\,dx'=1.
\end{equation}
In \cite[Lemma 5.2 and Corollary 5.5]{FT12}, the following result is proved.
\begin{Proposition}(\cite{FT12})\label{prop:lim_widehat_u}
 For every sequence $\e_{n}\to 0^+$, there 
   exist a subsequence
  $\{\e_{n_k}\}_k$ and $\widehat C\in\R\setminus\{0\}$ such that
  $\widehat u_{\e_{n_k}}\to \widehat C\widehat \Phi$ 
strongly in $\mathcal
  H_r$ for every $r>1$, in $C^2(\overline{B_{r_2}^-\setminus B_{r_1}^-})$
  for all $1<r_1<r_2$, and in $C^2(\{(x_1,x'):t_1\leq x_1\leq t_2,\,
  |x'|\leq 1\})$
  for all $0<t_1<t_2$,
where $\widehat
  \Phi$ is the unique solution to \eqref{eq_Phi_hat}. 
\end{Proposition}

\noindent
The following lemma ensures that the constant $\widehat C$ in
Proposition \ref{prop:lim_widehat_u} is
positive.
\begin{Lemma}\label{l:hatCpositive}
  Let  $\{\e_{n}\}_n$, 
  $\{\e_{n_k}\}_k$ and $\widehat C$ as in 
Proposition \ref{prop:lim_widehat_u}. Then $\widehat C>0$.
\end{Lemma}
\begin{pf}
Let us assume by contradiction that $\widehat C<0$. Since $\widehat u_{\e_{n_k}}\to \widehat C\widehat \Phi$ 
strongly in $C^2(\Gamma_2^-)$, we have that, if $k$ is sufficiently
large, then  $u_{\e_{n_k}}<0$ on $\Gamma_{2\e_{n_k}}^-$.
 
From \cite[Corollary 1.3]{FT12}, there exists a sub-subsequence
$\{\e_{n_{k_j}}\}_j$  such that 
\[
2\e_{n_{k_j}}<\frac1j,\quad 
u_{\e_{n_{k_j}}}>0 \text{ on }\Gamma^-_{1/j},\quad 
\text{and}\quad
u_{\e_{n_{k_j}}}<0 \text{ on }\Gamma_{2\e_{n_{k_j}}}^-
.
\]
 Therefore, for $j$ large, the functions
\begin{equation*}
  v_j=\begin{cases}
    u_{\e_{n_{k_j}}},&\text{in }D^-\setminus B_{1/j}^-,\\
    u ^+_{\e_{n_{k_j}}},&\text{in } B_{1/j}^-\setminus B_{2\e_{n_{k_j}}}^-
,\\
    0,&\text{in } B_{2\e_{n_{k_j}}}^-\cup\{(x_1,x')\in \Omega^{\e_{n_{k_j}}}:\, x_1\geq0\},
  \end{cases}
\quad
\widetilde v_j := \dfrac{v_{j}}{\big(\int_{\R^N}p\,v_{j}^2dx\big)^{\!\frac12}}
\end{equation*}
are well-defined and belong to $\Di$ (if trivially extended to the whole $\R^N$).

Let $A_j=D^-\setminus B_{1/j}^-$. Then $\widetilde v_j$
satisfies
\begin{equation*}
\begin{cases}
-\Delta \widetilde v_j=\lambda_{\e_{n_{k_j}}} p \widetilde v_j,&\text{in }A_j,\\
\widetilde v_j=0,&\text{on }\partial A_j\cap\partial D^-,\\
\int_{\R^N}p\,\widetilde v_j^2\,dx=1.
\end{cases}
\end{equation*}
Testing equation \eqref{problema} for $\e=\e_{n_{k_j}}$ with $v_j$ we
obtain 
\[
\int_{D^-\setminus B_{2\e_{n_{k_j}}}^-}|\nabla v_j|^2dx= 
\lambda_{\e_{n_{k_j}}}\int_{D^-\setminus B_{2\e_{n_{k_j}}}^-}pv_j^2dx,
\] 
hence 
\begin{equation*}
  \int_{\R^N}p \widetilde v_j^2dx=\int_{D^-\setminus
    B_{2\e_{n_{k_j}}}^-}\!\!\!\!\!p \widetilde
  v_j^2dx=1,\ 
\int_{\R^N}|\nabla \widetilde
  v_j|^2dx=\int_{D^-\setminus B_{2\e_{n_{k_j}}}^-}\!\!\!\!\!|\nabla
  \widetilde v_j|^2dx= \lambda_{\e_{n_{k_j}}}.
\end{equation*}
Hence $\{\widetilde v_j\}_{j}$ is bounded in
$\Di$ and, along a subsequence, $\Di$-weakly converges to some
$\widetilde v\in\Di$ such that $\int_{\R^n}p\widetilde
v^2\,dx=1$, $\mathop{\rm supp}\widetilde v\subset D^-$, and 
\begin{equation*}
\begin{cases}
-\Delta \widetilde v=\lambda_{k_0}(D^+) p \widetilde v,&\text{in }D^-,\\
\widetilde v =0,&\text{on }\partial D^-,
\end{cases}
\end{equation*}
thus implying that $\lambda_{k_0}(D^+)\in \sigma_p(D^-)$ and contradicting assumption \eqref{eq:54}.
\end{pf}

We can conclude that the convergence in Proposition
\ref{prop:lim_widehat_u} is not up to subsequences,
since the limit can be univocally characterized by virtue of
\eqref{eq:normalization_eqhat} and Lemma \ref{l:hatCpositive}. Indeed,
passing to the limit in  \eqref{eq:normalization_eqhat}, we obtain that
\begin{equation*}
 \widehat C ^2= \left(\int_\Sigma \widehat\Phi^2(1,x')dx'\right)^{\!\!-1},
\end{equation*}
and hence by Lemma \ref{l:hatCpositive} we conclude that 
\begin{equation}\label{eq:widehat_C}
 \widehat C = \frac1{\sqrt{\int_{\Sigma}\widehat\Phi^2(1,x')dx'}}.
\end{equation}
Therefore we can improve Proposition \ref{prop:lim_widehat_u} as
follows.
\begin{Proposition}\label{prop:lim_widehat_u_improved}
As $\e\to0^+$, 
\[
\widehat u_{\e}\to
\frac{\widehat \Phi}{\sqrt{\int_{\Sigma}\widehat\Phi^2(1,x')dx'}}
\] 
strongly in $\mathcal H_r$ for every $r>1$,
in $C^2(\{(x_1,x'):t_1\leq x_1\leq t_2,\, |x'|\leq 1\})$ for all
$0<t_1<t_2$, and in
$C^2(\overline{B_{r_2}^-\setminus B_{r_1}^-})$ for all $1<r_1<r_2$, where $\widehat \Phi$ is the unique solution
to~\eqref{eq_Phi_hat}.
\end{Proposition}

As a further step in our analysis, we evaluate the asymptotic
behavior as $\e\to0^+$ of the function $\widetilde H_\e$ defined in
\eqref{Htilde} at $\e$-distance from the left junction in the
corridor. To this aim, the following lemma is required.
\begin{Lemma}\label{l:ashatPhi}
Let    $\widehat\Phi$ be the unique solution to
\eqref{eq_Phi_hat}. Then
\[
\widehat\Phi(x_1,x')=e^{\sqrt{\lambda_1(\Sigma)}x_1}\psi_1^\Sigma(x')
+O(e^{-\sqrt{\lambda_1(\Sigma)}\frac{x_1}2})
\] 
as $x_1\to+\infty$ uniformly with respect to $x'\in\Sigma$.
\end{Lemma}
\begin{proof}
Let $g:T_1^+\to\R$,
$g(x_1,x')=\widehat\Phi(x_1,x')-e^{\sqrt{\lambda_1(\Sigma)}x_1}\psi_1^\Sigma(x')$.
From \eqref{eq_Phi_hat} and \eqref{eq:prop_hat_phi} it follows that
$\int_{T_1^+}(|\nabla
g|^2+|g|^{2^*})<+\infty$, $g\geq 0$ in $T_1^+$, 
$g=0$ on $\{(x_1,x'):x_1>0,|x'|=1\}$, and $g$ weakly solves $-\Delta g=0$ in $T_1^+$.
Let 
$f(x_1,x')=e^{-\sqrt{\lambda_1(\Sigma)}\frac{x_1}2}\psi_1^\Sigma\big({x'}/2\big)$;
we notice that $f$
is harmonic and strictly positive in $T_1^+$, bounded from below away
from $0$ on $\{(x_1,x'):x_1=0,\,|x'|\leq1\}$, and $\int_{T_1^+}(|\nabla
f|^2+|f|^{2^*})<+\infty$. Hence, from the Maximum Principle we deduce
that $g(x)\leq {\rm const\,}f(x)$ in $T_1^+$, thus implying the
conclusion.
\end{proof}

We are now in position to provide the asymptotics of $\widetilde H_\e$
at $\e$-distance from the left junction in the corridor.
\begin{Proposition}\label{lemma_asintotico_H_epsilon_epsilon}
  Let $\widetilde H_\e$be as in \eqref{Htilde}. Then
 \begin{equation}
\label{asintotico_H_epsilon_epsilon}
\lim_{\e\to 0} \e^{-1} e^{\frac{\sqrt{\lambda_1(\Sigma)}}{\e}} \sqrt{\widetilde H_\e(\e)}
= \dfrac{\partial u_0}{\partial x_1}({\mathbf e_1})
\bigg(\int_{\Sigma}\Phi(1,x')\psi_1^\Sigma(x')dx'\bigg)
\sqrt{\int_{\Sigma}\widehat\Phi^2(1,x')dx'},
\end{equation}
being $u_0$ as in \eqref{eq:u0}, 
$\Phi$ the unique solution to \eqref{eq_Phi_1}, and $\widehat\Phi$ the unique solution to \eqref{eq_Phi_hat}.
\end{Proposition}
\begin{proof}
Let $\phi_\e$ as in \eqref{eq:phieps}, $C_\e$ as in
\eqref{eq_Veronica_phi}, and $A_\e,B_\e$ as in
\eqref{eq_Laura_phi}. We proceed by steps.

\smallskip\noindent
{\bf Step 1:} we claim that
\begin{equation}\label{primo_passo_lemma_asintotico_H_epsilon_epsilon}
  \lim_{\e\to0^+} \e^{-1}
  e^{-\frac{\sqrt{\lambda_1(\Sigma)}}{\e}(\e-1)} \phi_\e(\e)-
  \dfrac{B_\e}{\e} e^{-2\frac{\sqrt{\lambda_1(\Sigma)}}{\e}(\e-1)}=
  \dfrac{\partial u_0}{\partial x_1}({\mathbf e}_1) \, \int_{\Sigma}\Phi(1,x')\psi_1^\Sigma(x')dx'.
\end{equation}
To prove \eqref{primo_passo_lemma_asintotico_H_epsilon_epsilon}, let
us fix $x_0\in(0,1)$ and integrate \eqref{eq_Veronica_phi}
over $(\e,x_0)$.  For $\e$ sufficiently small, $x_0 >\e$. We obtain that
\begin{align*}
  \e^{-1} e^{-\frac{\sqrt{\lambda_1(\Sigma)}}{\e}(x_0-1)} \phi_\e(x_0)
  - &\e^{-1} e^{-\frac{\sqrt{\lambda_1(\Sigma)}}{\e}(\e-1)} \phi_\e(\e)\\
&  = \dfrac{C_\e}{2\sqrt{\lambda_1(\Sigma)}}
  \left(e^{-2\frac{\sqrt{\lambda_1(\Sigma)}}{\e}(\e-1)} -
    e^{-2\frac{\sqrt{\lambda_1(\Sigma)}}{\e}(x_0-1)}\right).
\end{align*}
Hence \eqref{primo_passo_lemma_asintotico_H_epsilon_epsilon} follows
from \eqref{eq:CepsBeps}, \eqref{ultimo_passo_asintotico_corridoio},
and \eqref{eq:asphieps}.

\smallskip\noindent
{\bf Step 2:}
we claim that, for every $h>0$,
\begin{equation}\label{secondo_passo_lemma_asintotico_H_epsilon_epsilon}
 e^{-h\sqrt{\lambda_1(\Sigma)}}\hat\phi(h) = \widehat C - \widehat C
 e^{-2h\sqrt{\lambda_1(\Sigma)}} +
\hat\phi(0)e^{-2h\sqrt{\lambda_1(\Sigma)}}
\end{equation}
where $\widehat C$ is as in \eqref{eq:widehat_C} and
\begin{equation}\label{eq:defphipicchat}
  \hat\phi:[0,+\infty)\to\R,\quad \hat\phi(t)=\widehat C
  \int_{\Sigma}\widehat \Phi(t, x')\psi_1^\Sigma(x')dx'.
\end{equation}
Since $\widehat \Phi$ is harmonic on its domain,
$\hat \phi$ solves
\begin{equation*}
\bigg(e^{-\sqrt{\lambda_1(\Sigma)}t}\hat\phi(t)\bigg)'
  =C e^{-2\sqrt{\lambda_1(\Sigma)}t},\quad\text{for
    all }t\in[0,+\infty),
\end{equation*}
being $C$ a real constant (independent of $t$). 
Integrating the above equation over $(0,h)$, we obtain
\[
e^{-h \sqrt{\lambda_1(\Sigma)}}\hat\phi(h)-\hat\phi(0)
 =
\dfrac{C}{2\sqrt{\lambda_1(\Sigma)}}\left(1-e^{-2h \sqrt{\lambda_1(\Sigma)}}\right)\
\text{for
all }h>0
\]
and then
\begin{equation}\label{eq:Cjbis}
C= \dfrac{2\sqrt{\lambda_1(\Sigma)}}{1-e^{-2h \sqrt{\lambda_1(\Sigma)}} }
\left(
e^{-h \sqrt{\lambda_1(\Sigma)}}\hat\phi(h)-\hat\phi(0)\right) \quad\text{for
all }h >0.
\end{equation}
From Lemma \ref{l:ashatPhi} it follows that
$e^{-h\sqrt{\lambda_1(\Sigma)}}\hat\phi(h)=\widehat C+o(1)$ as $h\to
+\infty$, then letting $h\to+\infty$ in \eqref{eq:Cjbis} we obtain that
$C=2\sqrt{\lambda_1(\Sigma)}\left(\widehat C - \hat\phi(0)\right)$ and
then
\[
e^{-h \sqrt{\lambda_1(\Sigma)}}\hat\phi(h)
 =\big(\widehat C - \hat\phi(0)\big)\Big(1-e^{-2h
    \sqrt{\lambda_1(\Sigma)}}\Big)
+\hat\phi(0)
\quad
\text{for
all }h>0
\]
which yields claim \eqref{secondo_passo_lemma_asintotico_H_epsilon_epsilon}.

\smallskip\noindent
{\bf Step 3:} we claim that
\begin{equation}\label{terzo_passo_lemma_asintotico_H_epsilon_epsilon}
 \lim_{\e\to 0^+} \dfrac{B_\e e^{\frac{\sqrt{\lambda_1((\Sigma)}}{\e}}}{\big(\widetilde H_\e(\e)\big)^{1/2}}
= \hat\phi(0)-\widehat C.
\end{equation}
To prove \eqref{terzo_passo_lemma_asintotico_H_epsilon_epsilon} we follow
the scheme of Proposition \ref{asintotico_H_epsilon_canale}.
We fix $k>1$ and integrate  \eqref{eq_Veronica_phi} over $(\e,k\e)$,
thus obtaining
\begin{multline*}
  e^{-\frac{\sqrt{\lambda_1(\Sigma)}}{\e} (k\e-1)}\phi_\e(k\e)
  - e^{-\frac{\sqrt{\lambda_1(\Sigma)}}{\e} (\e-1)}\phi_\e(\e)\\
  = \dfrac{C_\e \e}{2\sqrt{\lambda_1(\Sigma)}}
  \left(e^{-2\frac{\sqrt{\lambda_1(\Sigma)}}{\e}(\e-1)} -
    e^{-2\frac{\sqrt{\lambda_1(\Sigma)}}{\e}(k\e-1)}\right)
\end{multline*}
i.e., in view of \eqref{eq:CepsBeps},
\begin{equation*} 
e^{-\sqrt{\lambda_1(\Sigma)}
    k}\dfrac{\phi_\e(k\e)}{\sqrt{\widetilde H_\e(\e)}} -
  e^{-\sqrt{\lambda_1(\Sigma)}}\dfrac{\phi_\e(\e)}{\sqrt{\widetilde H_\e(\e)}}
  = \dfrac{B_\e e^{\frac{\sqrt{\lambda_1(\Sigma)}}{\e}}}{\sqrt{\widetilde
      H_\e(\e)}} \left(e^{-2k\sqrt{\lambda_1(\Sigma)}} -
    e^{-2\sqrt{\lambda_1(\Sigma)})}\right),
\end{equation*}
which, in view of \eqref{eq:phieps}, Proposition
\ref{prop:lim_widehat_u_improved}, \eqref{eq:defphipicchat}, and 
\eqref{secondo_passo_lemma_asintotico_H_epsilon_epsilon}, implies that
\begin{align*}
\dfrac{B_\e e^{\frac{\sqrt{\lambda_1(\Sigma)}}{\e}}}{\sqrt{\widetilde
    H_\e(\e)}} \left(e^{-2k\sqrt{\lambda_1(\Sigma)}} -
  e^{-2\sqrt{\lambda_1(\Sigma)})}\right)
&\to e^{-\sqrt{\lambda_1(\Sigma)} k} \hat\phi(k) - e^{-\sqrt{\lambda_1(\Sigma)}} \hat\phi(1) \\
  &\quad= \Big(e^{-2k\sqrt{\lambda_1(\Sigma)}} - e^{-2\sqrt{\lambda_1(\Sigma)})}\Big) \big(\hat\phi(0)-\widehat C\big)
\end{align*}
as $\e\to0^+$, thus proving claim \eqref{terzo_passo_lemma_asintotico_H_epsilon_epsilon}.

In order to conclude the proof, we observe that from
\eqref{primo_passo_lemma_asintotico_H_epsilon_epsilon} it follows
\begin{multline*}
\lim_{\e\to 0^+}  \sqrt{\widetilde H_\e(\e)} \e^{-1}
  e^{\frac{\sqrt{\lambda_1(\Sigma)}}{\e}} \left(
    e^{-\sqrt{\lambda_1(\Sigma)}}\dfrac{\phi_\e(\e)}{\sqrt{\widetilde
        H_\e(\e)}}
    - e^{-2\sqrt{\lambda_1(\Sigma)}}\dfrac{B_\e e^{\frac{\sqrt{\lambda_1(\Sigma)}}{\e}}}{\sqrt{\widetilde H_\e(\e)}} \right)\\
 =\frac{\partial u_0}{\partial x_1}({\mathbf e}_1) \,
  \int_{\Sigma}\Phi(1,x')\psi_1^\Sigma(x')dx'
\end{multline*}
which is sufficient to conclude
in view of \eqref{eq:phieps}, \eqref{eq:121}, Proposition
\ref{prop:lim_widehat_u_improved}, \eqref{eq:defphipicchat}, 
\eqref{terzo_passo_lemma_asintotico_H_epsilon_epsilon}
and~\eqref{secondo_passo_lemma_asintotico_H_epsilon_epsilon}.
\end{proof}

\noindent
We are now in position to derive an asymptotics for the normalization $(\int_{\Gamma^-_{\tilde
k}}u_{\e_{n_j}}^2d\sigma)^{1/2}$.
\begin{Proposition}\label{lemma_asintotico_H_epsilon_meno_R}
Let $\tilde k$ as in Theorem \ref{t:main}. Then 
\begin{multline}\label{eq:3}
   \lim_{\e\to0^+} e^{\frac{\sqrt{\lambda_1(\Sigma)}}{\e}}\e^{-N}\sqrt{\int_{\Gamma_{\widetilde
          k}^-} u_{\e}^2\,d\sigma}\\=
\sqrt{\int_{\Gamma^-_{\tilde
          k}}\overline{U}^2d\sigma}
\bigg(  \int_{\SN_-}\widehat \Phi(\theta)\Psi^-(\theta)d\sigma(\theta)\bigg)
\bigg(\int_{\Sigma}\Phi(1,x')\psi_1^\Sigma(x')dx'\bigg)
\dfrac{\partial u_0}{\partial x_1}({\mathbf e_1}).
\end{multline}
\end{Proposition}
\begin{proof}
We divide the proof in several steps.

\smallskip\noindent
{\bf Step 1:} we claim that,
for every $h>k>1$,
\begin{equation}\label{eq_primo_passo_asintotico_H_meno_R}
\frac{\hat v(h)}{h^{1-N}}=\frac{\hat v(k)}{k^{1-N}}=\hat v(1)
\end{equation}
where
\begin{equation*}%\label{eq:defvhat}
  \hat v:[1,+\infty)\to\R,\quad \hat v(r):=\widehat C
  \int_{\SN_-}\widehat \Phi(r\theta)\Psi^-(\theta)d\sigma(\theta),
\end{equation*}
with $\widehat \Phi$ being the unique solution to \eqref{eq_Phi_hat}
and $\widehat C$ as in \eqref{eq:widehat_C}.
Since $\widehat \Phi$ is harmonic on its domain,
$\hat v$ solves
\begin{equation*}
 \bigg(r^{N+1}\Big(\dfrac{\hat v}{r}\Big)'\bigg)'=0,\quad\text{in }(1,+\infty),
\end{equation*}
hence, by integration, there exists $C\in\R$ (independent of $h,k$) such that 
\begin{equation*}
 \dfrac{\hat v(h)}{h} =  \dfrac{\hat v(k)}{k}+\dfrac{C}{N}(k^{-N}-h^{-N}),\quad\text{for all }1<k<h.
\end{equation*}
Hence 
\[C= \frac{N}{k^{-N}-h^{-N}} \bigg( \frac{\hat v(h)}{h} -\frac{\hat
  v(k)}{k} \bigg),\quad\text{for all }
1<k<h.
\]
From \eqref{eq:stimahatPhiinfty}, it follows that $\frac{\hat v(h)}{h}
\to 0$ as $h\to +\infty$. Hence $C= -N \frac{\hat v(k)}{k^{1-N}}$ and claim
\eqref{eq_primo_passo_asintotico_H_meno_R} is proved.

\smallskip\noindent
{\bf Step 2:} for all $r\in(\e,3)$ let us define
\begin{equation*}
  \varphi_\e^-(r)=\int_{\SN_-}u_\e(r\theta)\Psi^-(\theta)\,d\sigma(\theta). 
\end{equation*}
From \eqref{problema} and \eqref{eq:p2} it follows that $\varphi_\e^-$
satisfies 
\begin{equation}\label{eq:1phiepsmeno}
\bigg(r^{N+1}\Big(\frac{\varphi_{\e}^-(r)}{r}\Big)'\bigg)'=0, \quad
\text{in } (\e,3),
\end{equation}
and 
hence there exists a constant $d_{\e}$ (depending on $\e$ but
independent of $r$) such that 
\begin{equation}\label{eq:eqdeps}
\Big(\frac{\varphi_{\e}^-(r)}{r}\Big)'=\frac{d_{\e}}{r^{N+1}}, \quad
\text{in } (\e,3).
\end{equation}
 We claim that, for every $k>1$,
\begin{equation}\label{eq_secondo_passo_asintotico_H_meno_R}
  \lim_{\e\to 0^+}\dfrac{d_\e}{N \e^{N-1} \sqrt{\widetilde H_\e(\e)} }=- \frac{\hat v(k)}{k^{1-N}}.
\end{equation}
Integration of  \eqref{eq:eqdeps} in $(k\e,h\e)$ for $h>k>1$ yields
\begin{equation}\label{eq:hkeps}
 \dfrac{\varphi_{\e}^-(h\e)}{h} - \frac{\varphi_{\e}^-(k\e)}{k} =
 \dfrac{d_{\e}}{N\e^{N-1}} \left(k^{-N}-h^{-N}\right),
\quad\text{for all }1<k<h<\frac3\e,
\end{equation}
and then
\begin{equation}\label{eq:depsas}
\dfrac{d_\e}{N\e^{N-1}\sqrt{\widetilde H_\e(\e)}}
= \dfrac{1}{k^{-N}- h^{-N}} \left( \dfrac{1}{h} \dfrac{\varphi_\e^-(h\e)}{\sqrt{\widetilde H_\e(\e)}}
- \frac1k\dfrac{\varphi_\e^-(k\e)}{\sqrt{\widetilde H_\e(\e)}} \right). 
\end{equation}
Since $\frac{\varphi_\e^-(r\e)}{\sqrt{\widetilde
    H_\e(\e)}}=\int_{\SN_-}\widehat
u_\e(r\theta)\Psi^-(\theta)\,d\sigma(\theta)$ for all $r>1$, from Proposition
\ref{prop:lim_widehat_u_improved} it follows that 
\[
\lim_{\e\to 0^+}\frac{\varphi_\e^-(r\e)}{\sqrt{\widetilde
    H_\e(\e)}}=\hat v(r), \quad\text{for all }r> 1,
\]
hence passing to the limit as $\e\to 0^+$ in \eqref{eq:depsas} we
obtain 
\[
\lim_{\e\to 0^+}\dfrac{d_\e}{N\e^{N-1}\sqrt{\widetilde H_\e(\e)}}
= \dfrac{1}{k^{-N}- h^{-N}} \bigg( \frac{\hat v(h)}{h}-\frac{\hat v(k)}k\bigg). 
\]
which yields claim \eqref{eq_secondo_passo_asintotico_H_meno_R} in
view of \eqref{eq_primo_passo_asintotico_H_meno_R}.

\smallskip\noindent
{\bf Step 3:} we claim that 
\begin{equation}\label{eq:2}
  \lim_{\e\to0^+}\frac{d_{\e}}{\sqrt{\int_{\Gamma_{\widetilde
          k}^-} u_{\e}^2\,d\sigma}}=-\frac{N}{\sqrt{\int_{\Gamma^-_{\tilde
          k}}\overline{U}^2d\sigma}}.
\end{equation}
From \eqref{eq:1phiepsmeno} it follows that there exist
$\alpha_\e,\beta_\e\in\R$ (depending on $\e$ but independent of $r$) such that
\[ 
\varphi_\e^-(r)=\alpha_\e r +\beta_\e r^{1-N},\quad\text{for all }r\in (\e,3).
\]
From \eqref{eq:eqdeps} it follows that
\begin{equation}\label{eq:betaepsdeps}
\beta_\e=-\frac{d_\e}{N}.
\end{equation}
From Proposition \ref{p:conUeps_nosottosucc} we have that 
\[
\frac{u_\e(x)}{\sqrt{\int_{\Gamma_{\widetilde k}^-} u_\e^2\,d\sigma}
} \to 
 \frac{\overline{U}}{\sqrt{\int_{\Gamma^-_{\tilde k}}\overline{U}^2d\sigma}}
\]
as $\e\to0^+$,
strongly in $\mathcal
  H_t^-$ for every $t>0$
and in $C^2(\overline{B_{t_2}^-\setminus
  B_{t_1}^-})$ for all $0<t_1<t_2$.
Hence, for all $r\in(0,3)$, 
\begin{equation}\label{eq:asphiepsmenonorm}
\lim_{\e\to0^+}\frac{\varphi_\e^-(r)}{\sqrt{\int_{\Gamma_{\widetilde
        k}^-} u_\e^2\,d\sigma}}=
\frac{\overline{\varphi}(r)}{\sqrt{\int_{\Gamma^-_{\tilde k}}\overline{U}^2d\sigma}}
\end{equation}
where 
\[
\overline{\varphi}(r):=\int_{\SN_-}\overline{U} (r\theta)\Psi^-(\theta)\,d\sigma(\theta). 
\]
Since $\overline{U}$ is harmonic in $B_3^-$, it is easy to prove that there exist
$a,b \in\R$ such that
\[ 
\overline{\varphi}(r)=a \,r +b  \,r^{1-N},\quad\text{for all }r\in (0,3).
\]
From \eqref{problema_U} it follows that $b =1$.
Hence \eqref{eq:asphiepsmenonorm} can be rewritten as
\begin{equation}\label{eq:1}
\lim_{\e\to0^+}\frac{\alpha_\e r +\beta_\e r^{1-N}}{\sqrt{\int_{\Gamma_{\widetilde
        k}^-} u_\e^2\,d\sigma}}=
\frac{a \,r +r^{1-N}}{\sqrt{\int_{\Gamma^-_{\tilde
        k}}\overline{U}^2d\sigma}},\quad
\text{for all $r\in(0,3)$}.
\end{equation}
We claim that $\frac{\alpha_\e}{\beta_\e}=O(1)$ as $\e\to0^+$; to
prove this, we assume by contradiction that along a sequence
$\e_n\to0^+$ there holds $\lim_{n\to+\infty}
\frac{\beta_{\e_n}}{\alpha_{\e_n}}=0$. Then from \eqref{eq:1}
there would follow, for all $r\in(0,3)$,
\begin{equation*}
\lim_{n\to+\infty}\frac{\alpha_{\e_n} }
{\sqrt{\int_{\Gamma_{\widetilde
        k}^-} u_{\e_n}^2\,d\sigma}}=
\lim_{n\to+\infty}\frac{\alpha_{\e_n} r
  +
\beta_{\e_n} r^{1-N}}
{\sqrt{\int_{\Gamma_{\widetilde
        k}^-} u_{\e_n}^2\,d\sigma}}\frac{1}{\big(r
  +\frac{\beta_{\e_n}}{\alpha_{\e_n}} r^{1-N}\big)}=
\frac{a  +r^{-N}}{\sqrt{\int_{\Gamma^-_{\tilde
        k}}\overline{U}^2d\sigma}},
\end{equation*}
thus giving rise to a contradiction since different values of $r$
yield different limits for the same sequence.

From the fact that $\{\frac{\alpha_\e}{\beta_\e}\}_\e$ is bounded, it
follows that there exist a sequence
$\e_n\to0^+$ and some $\ell\in\R$ such that $\lim_{n\to+\infty}
\frac{\alpha_{\e_n}}{\beta_{\e_n}}=\ell$. Hence \eqref{eq:1} implies
that 
\begin{equation*}
  \lim_{n\to+\infty}\frac{\beta_{\e_n}}{\sqrt{\int_{\Gamma_{\widetilde
          k}^-} u_{\e_n}^2\,d\sigma}}=
  \lim_{n\to+\infty}\frac{\alpha_{\e_n} r + \beta_{\e_n} r^{1-N}}
  {\sqrt{\int_{\Gamma_{\widetilde k}^-}
      u_{\e_n}^2\,d\sigma}}\frac{1}{\big(\frac{\alpha_{\e_n}}{\beta_{\e_n}}r
    + r^{1-N}\big)}
  =\frac{1}{\sqrt{\int_{\Gamma^-_{\tilde
          k}}\overline{U}^2d\sigma}}\frac{a \,r
    +r^{1-N}}{\big(\ell r + r^{1-N}\big)}
\end{equation*}
for all $r\in(0,3)$, hence necessarily $\ell=a$ (otherwise  different values of $r$
would yield different limits for the same sequence). In particular the
limit $\lim_{n\to+\infty}
\frac{\alpha_{\e_n}}{\beta_{\e_n}}$ does not depend on the sequence
$\{\e_{n}\}_n$, thus
implying that 
\[
\lim_{\e\to0^+}
\frac{\alpha_{\e}}{\beta_{\e}}=a. 
\]
 Hence \eqref{eq:1} implies that, for all $r\in(0,3)$,
\begin{align*}
  \lim_{\e\to0^+}\frac{\beta_{\e}}{\sqrt{\int_{\Gamma_{\widetilde
          k}^-} u_{\e}^2\,d\sigma}}&=
  \lim_{\e\to0^+ }\frac{\alpha_{\e} r + \beta_{\e} r^{1-N}}
  {\sqrt{\int_{\Gamma_{\widetilde k}^-}
      u_{\e}^2\,d\sigma}}\frac{1}{\big(\frac{\alpha_{\e}}{\beta_{\e}}r
    + r^{1-N}\big)}\\
 & =\frac{1}{\sqrt{\int_{\Gamma^-_{\tilde
          k}}\overline{U}^2d\sigma}}\frac{a \,r
    +r^{1-N}}{\big(a r + r^{1-N}\big)}=\frac{1}{\sqrt{\int_{\Gamma^-_{\tilde
          k}}\overline{U}^2d\sigma}}
\end{align*}
which yields claim \eqref{eq:2} in view of \eqref{eq:betaepsdeps}.

\smallskip\noindent
Combining \eqref{eq:2}, \eqref{eq_secondo_passo_asintotico_H_meno_R},
and Proposition \ref{lemma_asintotico_H_epsilon_epsilon}, we finally
obtain
\begin{multline*}
   \lim_{\e\to0^+} e^{\frac{\sqrt{\lambda_1(\Sigma)}}{\e}}\e^{-N}\sqrt{\int_{\Gamma_{\widetilde
          k}^-} u_{\e}^2\,d\sigma}\\=
\sqrt{\int_{\Gamma^-_{\tilde
          k}}\overline{U}^2d\sigma}
\bigg(  \int_{\SN_-}\widehat \Phi(\theta)\Psi^-(\theta)d\sigma(\theta)\bigg)
\bigg(\int_{\Sigma}\Phi(1,x')\psi_1^\Sigma(x')dx'\bigg)
\dfrac{\partial u_0}{\partial x_1}({\mathbf e_1})
\end{multline*}
thus completing the proof.
\end{proof}

\begin{remark}
 We would like to stress that, in the flavor of \cite{AT12}, the asymptotic behavior of solutions
is affected by the domain's geometry: the constants at the left hand
side of \eqref{eq:3} depend on the solutions
of the relative blow-up limits, in addition to the initial normalization $u_0$.
It is interesting to notice that the geometry of the left-hand side already
appears in the asymptotics of Proposition \ref{lemma_asintotico_H_epsilon_epsilon}, even if the solution
has not crossed the left junction yet.
\end{remark}

\begin{pfn}{Teorema \ref{t:teorema_principale}}
It follows combining Propositions \ref{p:conUeps_nosottosucc} and \ref{lemma_asintotico_H_epsilon_meno_R}.
\end{pfn}

%%%%%%%%%%%%%%%%%%%%%%%%%%%%%%%%%%%%%%%%%%%%%%%%%%%%%%%%%%%%%%%%%%%%%%%%%%%%%%%%%%%%%%%%%%%%%%%%%%%%%%%


\begin{thebibliography}{99}

\bibitem{AT12} L. Abatangelo, S. Terracini, {\em Positive harmonic functions in union of chambers\/}.
Preprint 2012, {\tt ArXiv: 1210.1070}.

\bibitem{almgren} F. J. Jr. Almgren, {\it $Q$ valued functions
    minimizing Dirichlet's integral and the regularity of area
    minimizing rectifiable currents up to codimension two},
  Bull. Amer. Math. Soc. 8 (1983), no. 2, 327--328.

\bibitem{anne} C. Ann\'e, {\it 
Fonctions propres sur des vari\'et\'es avec des anses fines,
application \`a la multiplicit\'e}, 
Comm. Partial Differential Equations 15 (1990), no. 11, 1617--1630.

\bibitem{AD} W. Arendt, D. Daners, {\it Uniform convergence for elliptic problems on varying domains}, Math. Nachr. 280 (2007), no. 1-2, 28--49.

\bibitem{arrieta} J. M. Arrieta, {\it Domain Dependence of Elliptic
    Operators in Divergence Form}, Resenhas 3 (1997), No. 1, 107--123.

\bibitem{AK} J. M. Arrieta, D. Krej{\v{c}}i{\v{r}}{\'{\i}}k,
{\it Geometric versus spectral convergence for the Neumann Laplacian
  under exterior perturbations of the domain},
 Integral methods in science and engineering. Vol. 1, 9–19,
 Birkh\"auser Boston, 
Inc., Boston, MA, 2010. 

\bibitem{BV} I. Babu{\v{s}}ka, R. V{\'y}born{\'y}, {\it Continuous dependence of
    eigenvalues on the domain},
 Czechoslovak Math. J. 15(90) (1965), 169--178.


\bibitem{BHM} R. Brown, P. D.  Hislop, A. Martinez, {\it Lower bounds
    on the interaction between cavities connected by a thin tube},
  Duke Math. J. 73 (1994), no. 1, 163--176.

\bibitem{bucur2006} D. Bucur, {\it Characterization of the shape
    stability for nonlinear elliptic problems},
 J. Differential Equations 226 (2006), no. 1, 99--117. 

\bibitem{BZ} D. Bucur, J.P. Zol\'esio, {\it Spectrum stability of an elliptic operator to domain perturbations}, J. Convex Anal. 5 (1998), no. 1, 19--30.

\bibitem{CF} C. Cacciapuoti, D. Finco, {\it  Graph-like models for
    thin Waveguides with Robin boundary
 conditions}, Asymp.An. 70  (2010), no. 3--4, 199--230.

\bibitem{CH} R. Courant, D. Hilbert, {\it Methods of Mathematical
    Physics}, Vol I, Wiley-Interscience, New York (1953). German
  Edition (1937).
  
\bibitem{DD} E.N.  Dancer, D. Daners, {\it Domain perturbation for elliptic equations subject to Robin boundary conditions}, J. Differential Equations 138 (1997), no. 1, 86--132.
  
\bibitem{Dancer1} E. N. Dancer, {\it The Effect of Domain Shape on the
    Number of Positive Solutions of Certain Nonlinear Equations},
  J. Differential Equations 74 (1988), 120--156.
 
\bibitem{Dancer2} E. N. Dancer, {\it The Effect of Domain Shape on the
    Number of Positive Solutions of Certain Nonlinear Equations, II}
  J. Differential Equations 87 (1990), 316--339.

\bibitem{daners} D. Daners, {\it Dirichlet problems on varying domains},
J. Differential Equations 188 (2003), no. 2, 591--624.


\bibitem{FT12} V. Felli, S. Terracini, {\em Singularity of
    eigenfunctions at the junction of shrinking tubes. Part I\/},
  Preprint 2012, {\tt arXiv 1202.4414}.

\bibitem{GL} N. Garofalo, F.-H.  Lin, {\it Monotonicity properties of
    variational integrals, $A\sb p$ weights and unique continuation},
  Indiana Univ. Math. J.  35 (1986), no. 2, 245--268.

\bibitem{grigorieff} R. D. Grigorieff, {\it Diskret kompakte
    Einbettungen in Sobolewschen Räumen}, 
Math. Ann. 197 (1972) 71–85.

\bibitem{kuchment} P. Kuchment, {\it Quantum graphs. II. Some spectral
    properties of quantum and combinatorial graphs}, J. Phys. A 38
  (2005), no. 22, 4887--4900.
  

\bibitem{RT} J. Rauch, M. Taylor, {\it Potential and scattering theory
    on wildly perturbed domains},  J. Funct. Anal. 18 (1975), 27–59. 

\bibitem{stummel} F. Stummel, {\it Diskrete Konvergenz linearer
    Operatoren. I}, Math. Ann. 190 (1970/71), 45–92.








\end{thebibliography}
\end{document}